\documentclass[a4paper,12pt]{article}
 \usepackage{amssymb,amsthm,amsmath,amsfonts}
\usepackage[left=2.5cm,right=2.5cm,top=2cm,bottom=2cm]{geometry}
\usepackage{cite}
 \usepackage{color}
\newtheorem{thm}{Theorem}[section]
\newtheorem{cor}[thm]{Corollary}
\newtheorem{lem}[thm]{Lemma}
\newtheorem{prop}[thm]{Proposition}
\theoremstyle{definition}
\newtheorem{defn}[thm]{Definition}
\theoremstyle{remark}
\newtheorem{rem}[thm]{Remark}
\newtheorem*{ex}{Example}
\numberwithin{equation}{section}               
\begin{document}

\date{}
\title{Some global operators and  the material derivative}
\author{J. O. Gonz\'alez-Cervantes$^{(1)}$, D. Gonz\'alez-Campos$^{(1)}$,\\ Juan Bory-Reyes$^{(2)}$}
\vskip 1truecm

\date{\small $^{(1)}$ Departamento de Matem\'aticas, ESFM-Instituto Polit\'ecnico Nacional. 07338, Ciudad M\'exico, M\'exico\\ Email: jogc200678@gmail.com, daniel\_uz13@hotmail.com\\ $^{(2)}$ {SEPI, ESIME-Zacatenco-Instituto Polit\'ecnico Nacional. 07338, Ciudad M\'exico, M\'exico}\\Email: juanboryreyes@yahoo.com}

\maketitle
\begin{abstract} 
\noindent
  
  The theory of the operator   
$$G(x) = |\underline{x}|^2  \frac{\partial }{\partial x_0} + \underline{x}  \sum_{j=1}^n x_j  \frac{\partial }{\partial  x_j} $$ 
   is deeply associated with the  slice monogenic function theory  and   has grown in recent years. In particular,  for $n=3$ the quaternionic version of $G$ has been recently used to study the quaternionic slice regular function theory.

This work extends the study of the $G$ operator in two senses:

a) Clifford's analysis structure. The function theory induced by   the operator 
\begin{align*}
	\mathcal H_a (x)  = {\underline a} ( {x})  \frac{\partial  }{\partial x_0} - \sum_{i=1}^n \left( \sum_{j=1}^n a_j ( {x}) \frac{\partial (a^{-1})_i}{\partial  y_j}\circ  a ( {x}) \right) \frac{\partial}{\partial x_i},  
\end{align*} 
 where $a$ is a function with certain properties with domain  in $\mathbb R^{n+1}$ is presented extending the already known results of the $G$ (if $a=I$ is the identity function   
 we can see that  $G=\mathcal H_I$) such as Splitting  Lemma,  Representations Theorem,  Cauchy formula,  a characterization of the zero sets of $H_a$ and  
 a development in  power series.  
 
Also  some properties  of  the  material derivative:
\begin{align*}
{D}_u =  \frac{\partial }{\partial x_0} + \sum_{j=1}^n u_j \frac{\partial  }{\partial x_j}, 
\end{align*}
where  $u $ is a certain $\mathbb R^{n}$-valued function  with domains in $\mathbb R^{n}$, or $\mathbb R^{n+1}$, are presented  as   consequences of function theory induced by      $\mathcal H_a$.

	b) Structure of quaternionic analysis. In particular, the case $n=3$  is approached from the point of view of quaternionic analysis presenting results such as  Splitting  Lemma,  Representations Theorem,   Borel-Pompeiu, Stokes and  Cauchy formulas, a  conformal covariant property  and  a characterization of the zero sets of the quaternonic version of $\mathcal H_a$ and its consequences for $D_u$.

 \end{abstract}

\vspace{0.3 cm}

\small
\noindent
{\textbf{Keywords.} Nonconstant coefficient differential operator, Clifford and Quaternionic analysis, Borel-Pompeiu and Cauchy formulas, Conformal covariant property, material derivatives.}  

\section{Introduction}

The properties of the  operator 
$$G(x) = |\underline{x}|^2  \frac{\partial }{\partial x_0} +  \underline{x}  \sum_{j=1}^n x_j  \frac{\partial }{\partial  x_j},$$
where $\underline{x}$ is the 1-vector part of the paravector $x = x_0 + \underline{x}$ and $n \in \mathbb N$ are well known in the theory of Clifford analysis, particularly   in
  the theory of slice monogenic functions,    
see  \cite{A1, A2, A3, CS, CZ, GlobalOp, GP}.  Despite the fact that $G$ has nonconstant coefficients, the kernel of this operator is associated to two different definitions of slice monogenic functions that turn out to be equivalent under suitable conditions on the domains on which they are defined.

The   material derivative 
\begin{align*}
{D}_u =  \frac{\partial }{\partial x_0} + \sum_{j=1}^n u_j \frac{\partial  }{\partial x_j}, 
\end{align*}
where  $u $ is a certain $\mathbb R^{n}$-valued function  with domains in $\mathbb R^{n}$, or $\mathbb R^{n+1}$,  is the rate of change over time of a tensor field  $\displaystyle y$ (temperature, velocity, stress tensor, or other physical characteristic) with the sense that it depends only on position and time coordinates; see \cite{B,PK,Z}.
 The multidimensional material derivative has many applications, see \cite{F}. Among some important applications of the material derivative are $(n+1)$-incompressible Navier-Stokes equation system: 
\begin{align*}
	& {  D}_u(u) - \alpha \sum_{j=1}^n   \frac{\partial^2 u}{\partial x_j^2} = - \nabla p+ f(x, t).
\end{align*}
Or equivalently, 
\begin{align*}
	(1  +{\underline a} )   \frac{\partial u}{\partial x_0}    
	- H_a(u)  - \alpha \sum_{j=1}^n   \frac{\partial^2 u}{\partial x_j^2 }  =-  \nabla p+ f(x, t),\end{align*}
where 
\begin{align*}
	H_a (x)  = & {\underline a} (\underline{x})  \frac{\partial  }{\partial x_0} - \sum_{i=1}^n \left( \sum_{j=1}^n a_j (\underline{x}) \frac{\partial (a^{-1})_i}{\partial  y_j}\circ  a (\underline{x}) \right) \frac{\partial}{\partial x_i},  \\
	{  {D}_u f} =&  (1  +{\underline a}) {\frac {\partial f}{\partial y_0}} - H_a(f),
	\end{align*}
and $a$  is a bijective function   with   specific properties that will be given in this work. 
Although applications of the $n-$dimensional incompressible Navier-Stokes equations for $n>3$ are more common in mathematics than in other areas of science, for $n\leq 3$, there are  some applications of these equations in a wide range of engineering and scientific fields, such as aerodynamics, weather forecasting, ocean fluid mechanics, and energy system design (see \cite{Bat, OHJ, GR} and the references given there).

Another interesting study is obtained in case $n=3$ in which  the theory of quaternionic right-linear operator     
$$G(x) = |\underline{x}|^2  \frac{\partial }{\partial x_0} +  \underline{x}  \sum_{j=1}^3 x_j  \frac{\partial }{\partial  x_j} ,$$
where $\underline{x}$ is  3-dimensional  vector part  and   $x = x_0 + \underline{x}\in \mathbb H$, 
is associated with the quaternionic slice regular function theory, see  \cite{G, GG1, GG2, GPG}. In addition,  an important  application of the $4-$dimensional material derivative occurs in Euler equations of gas-dynamics:
\begin{align*}
	\frac{\partial \rho }{\partial t}   + \nabla\cdot  (\rho u) =& 0 ,\\
	\rho   {  D}_u  ( u)  + \nabla p  =& \rho a ,\\
	{  D}_u (p)  + \gamma_p ( \nabla \cdot u ) =& 0 , 
\end{align*}
see \cite{CC} for more details.

 In general, this work  introduces the theory of a  new family of  operators
  extending     the already known theory of  $G$ in Clifford and   quaternionic analysis  and  shows new properties of the material derivative.

This paper presents some extension  of  Splitting Lemma,  of Representation theorem,  of power series development, of Cauchy formula and a characterizations of   $\textrm{Ker}(H_a)$ and its consequences  for $ D_u)$  in Clifford analysis framework. 
  Meanwhile,  
 some extensions of Borel-Pompeiu, Stokes, and Cauchy formulas, of power series development, of a conformal covariant property associated to the operators $H_a$ and ${D}_u$ and   characterizations of   $\textrm{Ker}(H_a)$ and its consequences for $D_u)$ in the  quaternionic analysis framework  are also presented by this work.

This paper is organized as follows: Subsection \ref{SMT} presents some basic properties  of  the slice monogenic functions and  results about the operator $G$ in Clifford analysis. 
Subsection \ref{QSRFT} shows some results of the quaternionic slice regular function theory and 
 properties of the quaternionic operator $G$. Section \ref{MR} 
   presents the main results grouped into two subsections:   
 \ref{CLF} and \ref{QFR} which are  the 
 Clifford   and  quaternionic analysis frameworks, respectively.

\section{Preliminaries}\label{prel}
\subsection{Slice monogenic function theory}\label{SMT}

In this subsection we remember  some results  presented in   \cite{GlobalOp, CSS9}.

The real Clifford algebra over $n$ is denoted by $\mathbb R_{0,n}$. Their  imaginary units are $e_1, \dots , e_n$ 
satisfy   $e_i e_j +e_j e_i = -2\delta_{ij}$.
An element of $\mathbb R_{0,n}$ is given by 
$\sum_{A}e_ A x _A$, where $A = i_1\cdots i_r$,  $i\in \{1, 2, \dots , n\}$, $ i_1 < \cdots < i_r$
is a multi-index, $e_A = e_{i_1}e_{i_2} \cdots e_{i_r}$, $e_{\emptyset} = 1$ and  $x_A \in\mathbb R$. Some specific elements
in $\mathbb R_{0,n}$ are identified with vectors in the Euclidean space $\mathbb R^n$:  
$(x_1, x_2, . . . , x_n) \in \mathbb R^n$  is identified with a so-called 1-vector  $\underline x = x_1e_1 + \cdots + x_ne_n$
in the Clifford
algebra.

The element $(x_0, x_1, . . . , x_n) \in\mathbb R^{n+1} $  will be identified with 
the  paravector
$x = x_0 + \underline{x} = x_0 +\sum_{j=1}^n x_je_j$.  
The norm and conjugation of $x \in \mathbb  R^{n+1}$ are defined as
$|x|^2 = x^2_0 + x^2_1 + \cdots  + x^2_n$ and 
$\overline{x} = x_0 -  \underline{x}$, respectively. The real part of $x$ is   $Re[x]=x_0$. A
function $f : U \subset \mathbb  R^{n+1}\to \mathbb R_{0,n}$
is treated as a function $f(x)$ of the paravector $x$.
The symbol $\mathbb S$ will denote the sphere of unit $1$-vectors in $\mathbb R^n$, i.e,  
$\mathbb S = \{\underline x = e_1x_1 + \cdots + e_nx_n \ \mid \ x^2_1
+ \cdots + x^2_n = 1\}$, or equivalently, the   $(n-1)$-dimensional sphere in $\mathbb  R^{n+1}$ lying on the hyperplane $x_0 = 0$ 
that can be identified with the Euclidean space $\mathbb R^n$. Given $I\in \mathbb S$, then the vector space $\mathbb R+I\mathbb R $ will be denoted by $\mathbb C_I$ and  an element belonging to $\mathbb C_I$ will be denoted by $u+Iv$, for $u, v \in\mathbb  R$. Observe that $\mathbb C_I$ equipped with the inherited operations from $\mathbb R_{0,n}$ is isomorphic to $\mathbb C$ as fields, since $I^2=-1$ for every $I \in \mathbb S$.

Given a paravector $x = x_0 + \underline{x} \in \mathbb  R^{n+1}$ let us set
$$ I_x =\left\{ \begin{array}{l}  \frac{ \underline{x} }{ | \underline{x} |} ,  \ \  \underline{x} 
	\neq 0, \\
	\textrm{any element of $\mathbb S$ otherwise},
\end{array}  \right.$$
so, by definition, we have $x = x_0 + I_x | \underline{x}| \in \mathbb C_{I_x}$. Therefore, 
$\mathbb R^{n+1} = \bigcup_{I\in \mathbb S} \mathbb C_I$.

\begin{defn}  (Slice monogenic functions). Let $U \subset \mathbb R^{n+1}$ be an open set and
	$f : U\to \mathbb R_{0,n}$ be a real differentiable function. Let $I \in \mathbb  S$ denote 
	$ U_I = \mathbb C_I \cap U$. Then we say that $f$ is a   slice monogenic function,
	or s-monogenic function, if for every $I\in\mathbb S$, we have
	$$
	\left( \frac{\partial} 
	{\partial u}+ I \frac{\partial} 
	{\partial v} \right) f \mid_{U_I} (u+Iv) = 0, \quad  \forall   u+Iv \in  U_I.$$
\end{defn}
Let $\mathcal{SM}(U)$ denotes the set of slice monogenic functions defined on $U$.

\begin{defn} (Slice domains). Let $U\subset\mathbb  R ^{n+1}$ be a domain. We say that $U$ is a slice domain (s-domain for short) if $U \cap \mathbb  R$ is nonempty and if $\mathbb C_I \cap  U$ is a domain in $\mathbb C_I$ for all $I \in \mathbb S$.
\end{defn}

\begin{defn} (Axially symmetric domains). Let $U\subset\mathbb  R ^{n+1}$. We say that $U$ is
	axially symmetric if, for all $u+Iv \in U$, the whole $(n-1)$-sphere 
	$$[u+Iv]:=\{u+Jv \ \mid \ J\in \mathbb S\} \subset U.$$
\end{defn}

\begin{defn} (Holomorphic maps of a paravector variable). Let $U\subset\mathbb  R ^{n+1}$  
	an axially symmetric open set and $\mathcal U\subset \mathbb R \times \mathbb R$ be such that $x = u + Iv \in  U$ for
	all $(u, v) \in \mathcal  U$. The  holomorphic maps of a paravector variable defined on
	$U$ are of the form $f(x) = \alpha(u, v) + I\beta (u, v)$,  where the $\mathbb R^n$-valued differentiable
	functions  $\alpha, \beta$ are such that $\alpha (u, v) = \alpha(u,-v)$,
	$\beta(u, v) = - \beta (u,-v)$ for all $(u, v) \in \mathcal U$ and   Cauchy-Riemann system
	\begin{align*}
		\frac{\partial \alpha}{\partial u}- \frac{\partial \beta}{\partial v} = 0, \quad 
		\frac{\partial \alpha}{\partial v}+ \frac{\partial \beta}{\partial u} = 0 
	\end{align*}
	  is satisfied.
	  
	  By  $\mathcal{HM}(U)$ mean  the set of holomorphic maps of a paravector variable on   $U$.

	\end{defn}

\begin{thm} \cite{CSS9} Let $U \subset \mathbb  R^{n+1}$ be an axially symmetric s-domain and 
	$f : U \to \mathbb R^{n+1}$ be a real differentiable function. Then 
	$f\in\mathcal { HM}(U)$ if and only if $f \in \mathcal{SM}(U)$.
\end{thm}

\begin{defn} Let $U \subset \mathbb  R^{n+1}$ be an open set and $f\in C^1(U, R^n)$. Define
	$$ G [f] (x) = |\underline{x}|^2  \frac{\partial f}{\partial x_0} (x)+  \underline{x}  \sum_{j=1}^n x_j  \frac{\partial f}{\partial  x_j} (x), \quad \forall x\in U.$$
 
 Denote $\mathcal {GM}(U) = C^1(U,R^n) \cap \textrm{Ker}(G)$. Moreover, we set 
$$\mathcal {GM}_p(U) := \{f \in \mathcal{GM}(U) \ \mid \  \alpha	\ \textrm{ and  } \  \beta  \ \textrm{  do not depend	on} \  I \in \mathbb S\},$$ 
where 
$$\alpha  (u,v):= \frac{1}{2} (f(u + Iv) + f(u - Iv)), \quad \beta (u,v) := \frac{I}{2} (f(u - Iv) - f(u + Iv))$$
and $x = u + Iv$, $u = \textrm{Re}(x)$, $v = |x|$.
\end{defn}	

\begin{rem} Given  an open set  $U \subset \mathbb  R^{n+1}$   and  $f\in C^1(U, R^n)$ the right version of $G$:
	$$ G _r [f] (x) =  |\underline{x}|^2\frac{\partial f}{\partial x_0} (x)  +   \sum_{j=1}^n x_j \frac{\partial f}{\partial  x_j} (x)   \underline{x} , \quad \forall x\in U,$$
	   induces a function theory profoundly similar to that induced by $ G$, see \cite{GlobalOp}.
\end{rem}

\begin{thm}\cite{GlobalOp} \label{teoHGM} \ {} 
	\begin{enumerate}
		\item  If $U\subset \mathbb R^{n+1}$ is an axially symmetric s-domain and $
		f \in \mathcal { SM}(U)$ then $f \in  \mathcal {GM}(U)$.
		
		\item  If $U\subset \mathbb R^{n+1}$ is an axially symmetric domain and 
		$f\in \mathcal{ HM}(U)$ then
		$f \in\mathcal{ GM}(U)$.
		\item Let $U\subset \mathbb R^{n+1}$ be an axially symmetric domain that does not intersect
		the real axis and let $f : U\to \mathbb R_{0,n}$ be a real differentiable function.  Then
		$f \in \mathcal {HM}(U)$ if and only if $f \in \mathcal {GM}_p(U)$.
	\end{enumerate}
\end{thm}

\begin{thm}\cite{CSS9}\label{Repth}
	(Representation formula). Let $U \subset \mathbb R^{n+1}$ be an axially symmetric
	s-domain and $f\in \mathcal {SM}(U)$. For any vector $x = u + I_x v \in  U$ the following
	formula holds:
	$$
	f(x) = \frac{1}{2} (1 - I_x I) f(u + Iv) + \frac{1}{2}  (1 + I_xI) f(u-Iv), \quad I\in \mathbb  S.
	$$
\end{thm}

\begin{lem}\label{Split} \cite{CSS9} (Splitting Lemma). Given $f\in \mathcal {SM}(U)$. For every $I = I_1 \in \mathbb S$, let $ I_2, 
	\dots , I_n$ be
	a completion to a basis of $\mathbb R_{0,n}$ satisfying the defining relations $I_rI_s + I_sI_r 
	= -2\delta_{r,s}$.
	Then there exist $2^{n-1}$ holomorphic functions $F_A : U_{I}= U \cap \mathbb C_I\to \mathbb C_I$ such that
	for every $z = u + Iv$,
$$f \mid_{U_I}(z) = \sum^{n-1}_{|A|=0}	F_A(z)I_A, \quad  I_A = I_{i_1}\cdots I_{i_s} ,$$
	where $A = i_1\dots i_s$ is a subset of $\{2, . . . , n\}$, with $i_1 < \dots < i_s$, or, when $|A| = 0$,  $I_{\emptyset} = 1$.
\end{lem}

\begin{prop}\cite{CSS9} \label{Power}If  $f\in \mathcal {SM}(\mathbb B^{n+1}(0,1))$  there exists a sequence $(a_n)_{n\geq 0}$ of elements of $\mathbb R_{0,n}$  such that 
	\begin{align*}  f  (x) = \sum_{n=0}^\infty x^n a_n,\quad \forall x\in\mathbb B^{n+1}(0,1).  
	\end{align*}
\end{prop}

\begin{thm}\label{CauchyFormSM} \cite{CSS9} (Cauchy formula). Let $ W \subset \mathbb R^{n+1}$ be an open set and let 
	$f \in \mathcal {HM}(W)$. Let $U$ be a bounded axially symmetric open set such that 
	$U \subset W$. Suppose that the boundary of $U\cap \mathbb  C_I$ consists of a finite number of rectifiable Jordan curves
	for any $I\in \mathbb  S$. If $x \in U$ then 
	$$
	f(x) = \frac{1}{2\pi} \int_{\partial (U\cap \mathbb C_I )}
	S^{-1}(s, x)ds_I f(s), $$
	where $ds_I = \frac{ds}{I}$ and
	$$
	S^{-1}(s, x) = -(x^2-  2x \textrm{Re}[s] +
	 |s|^2)^{-1}(x- \overline{s}). $$
	The previous integral does not depend on $U$ nor on the imaginary unit $I \in \mathbb S$.
\end{thm}

\begin{thm} \cite{GlobalOp} For	$x\in  [s]$, $S^1(\cdot, x)$ belongs to $\textrm{Ker}(G)$ and $S^{-1}(s,\cdot)$ is in the kernel of the right operator $G_R$.
\end{thm}

\begin{thm} \label{CauchyFormSM2} \cite{GlobalOp} Let $W \subset \mathbb R^{n+1}$ be an open set and $f \in \mathcal{GM}_p(U)$, where $U$ is a
	bounded axially symmetric open set such that $U\subset  W$ that does not intersect the real
	line. Suppose that the boundary of $U \cap  \mathbb C_I$ consists of a finite number of rectifiable
	Jordan curves for any $I \in\mathbb S$. Then, if $x\in U$, we have
	$$ f(x) = \frac{1}{2\pi}\int_{\partial(U\cap \mathbb C_I )}
	S^{-1}(s, x)ds_I f(s),$$
	where $ds_I = \frac{ds}{I}$  and the integral does not depend
	on $U$ nor on the imaginary unit $I\in\mathbb  S$.
\end{thm}

\subsection{Quaternionic slice regular function theory}\label{QSRFT}
This subsection shows a brief summary of the  slice regular function theory obtained from    \cite{cgs, CSS, GS, GlobalOp, G, GG1, GG2, GPG, gpr, gssbook}. 

A quaternion is given by $\displaystyle q=\sum_{k=0}^3x_ke_k$, where $x_k\in\mathbb R$, for  all $k$. The skew-field of quaternions is denoted by $\mathbb H$. The real part of $q\in \mathbb H$ is Re$(q)=x_0$  and its vector part is   $\underline{q} = x_1e_1 +  x_2e_2  + x_3e_3$  which is   usually identified  with    $ ( x_1, x_2, x_3)\in\mathbb R^3=:Vec(\mathbb H)$.  The quaternionic conjugation and the quaternionic modulus of  $q$ are $ \bar q :=x_0 - \underline{q} $ and  $\|q\|=\sqrt{ q\bar q }= \sqrt{x_0^2+ x_1^2+  x_2^2 + x_3^2 }$, respectively. 

Denote
\begin{align*}
\mathbb B^{4}(0,1)=& \{q\in \mathbb H \ \mid \ \| q\| <1\}, \\
 \mathbb S^{3} = & \{q\in \mathbb H \ \mid \ \| q\| = 1\}.
 \end{align*}

The unit sphere  in $\mathbb R^3$ is denoted by
$$ \mathbb S^2 = \{\underline{q} \in Vec(\mathbb H) \ \mid \ \|\underline{ q}\|=1 \}.$$

Note that for all  ${ {I}} \in \mathbb S^2$ one has $ {I}^2=-1$ and   $\mathbb C( {I}):=\{x+ {I}y  \mid  x,y\in\mathbb R\} \cong \mathbb C$ as fields.

\begin{rem}
	It is important to see that  the vector part of a quaternion $q$ is denoted as $\vec q$ or ${\bf q}$ in some papers, but in this work we have decided to denote it as $\underline{q}$ (or $I$ and $J$)  in order to use a single notation throughout this paper (to facilitate the notation of a unit vector  and its relationship with a complex plane).
\end{rem}

\begin{defn}  (Slice regular functions). Let $U \subset \mathbb H$ be an open set and	$f : U\to  \mathbb H$ be a real differentiable function. Let $ {I} \in \mathbb  S^2$, denote 
	$ U_{ {I}} = \mathbb C_{ {I} } \cap U$. Then we say that $f$ is a slice regular function,
	or s-regular function, if for every $ {I}\in\mathbb S^2$, we have
	$$
	\left( \frac{\partial} 
	{\partial u}+  {I} \frac{\partial} 
	{\partial v} \right) f \mid_{U_{ {I}}} (u+ {I}v) = 0, \quad  \forall   u+ {I}v \in  U_{ {I}}. $$
Let $\mathcal{SR}(U)$ denotes the set of slice regular functions defined on $U$.
\end{defn}

\begin{defn} (Slice domains). Let $U\subset\mathbb  H$ be a domain. We say that $U$ is a  slice domain (s-domain for short) if $U \cap \mathbb  R$ is nonempty and if $\mathbb C_{ {I}} \cap  U$ is a domain  in $\mathbb C_{ {I}}$ for all ${I} \in \mathbb S^2$.
\end{defn}

\begin{defn} (Axially symmetric domains). Let $U\subset\mathbb  H$. We say that $U$ is
	axially symmetric if, for all $u+ {I}v \in U$, the whole $(2)$-sphere $[u+ {I}v]:=\{u+ Jv \ \mid \  J\in \mathbb S^2  \}$ is contained
	in $U$.
\end{defn}

\begin{defn} (Holomorphic maps of a paravector variable). Let $U\subset\mathbb  H$  
	an axially symmetric open set and $\mathcal U\subset \mathbb R \times \mathbb R$ be such that $q = u +  {I}v \in  U$ for
	all $(u, v) \in \mathcal  U$. The (left) holomorphic maps of a paravector variable defined on
	$U$ are of the form $f(q) = \alpha(u, v) + {I}\beta (u, v)$,  where $\alpha, \beta$ are $\mathbb R^n$-valued differentiable
	functions, satisfying   $\alpha (u, v) = \alpha(u,-v)$,
	$\beta(u, v) = - \beta (u,-v)$,  for all $(u, v) \in \mathcal U$, and   the Cauchy-Riemann system
	\begin{align*}
		\frac{\partial \alpha}{\partial u}- \frac{\partial \beta}{\partial v} = 0, \quad 
		\frac{\partial \alpha}{\partial v}+ \frac{\partial \beta}{\partial u} = 0.
	\end{align*}
	We denote by $\mathcal{HR}(U)$ the set of holomorphic maps of a paravector variable on the
	open set $U$.
\end{defn}

\begin{defn} Let $U \subset \mathbb H$ be an open set and let 
	$f : U \to \mathbb H$ be a real
	differentiable function. Define 
	$$ G [f] (q) = |\underline{x}|^2  \frac{\partial f}{\partial x_0} (q)+  \underline{x}  \sum_{j=1}^3 x_j  \frac{\partial f}{\partial  x_j} (q), \quad \forall q= x_0 +\underline{x}\in U,$$
	for all $f\in C^1(\mathbb H, \mathbb H)$.

Set $\mathcal GR(U) := C^1(U, \mathbb H) \cap \textrm{Ker}(G)$. 
	\end{defn}	

\begin{rem}
According to the previous notations and hypothesis we see that    the right version of $G$ is given by 
	$$ G_r [f] (q) = |\underline{x}|^2  \frac{\partial f}{\partial x_0} (q)+    \sum_{j=1}^3 x_j  \frac{\partial f}{\partial  x_j} (q)    \underline{x} , \quad \forall q\in U.$$
\end{rem}

\begin{thm}\cite{GlobalOp} \label{HGquater}
Let $U \subset \mathbb H$ be an axially symmetric s-domain.
\begin{enumerate}
		\item 	If  $f \in \mathcal {SR}(U)$  then $f \in \mathcal {GR}(U)$.
		\item 	If $f \in \mathcal { HR}(U)$  then $f\in \mathcal{ GR}(U)$.
	\end{enumerate}
\end{thm}

\begin{defn}\label{def1} Reproducing functions.
\begin{align*}
		A (\tau , x ) = & \frac{    1}{2\pi^2}  \frac{ \underline{x} \overline{(\tau  - x )}  \underline{\tau}   }{\| \underline{x}  \|^2\|\tau  - x \|^4  \|\underline{\tau} \|^2  }  , \\
		B (\tau , x)= &  
		\frac{1}{\pi^2   } \frac{\underline{x} }{   \|\underline{x} \|^2 } \bigg[     \frac{  {\tau}  + 3 \overline{x }   - 4 \overline{{\tau} }      }{   \|\tau -x \|^4  }     + \frac{  4 (\overline{ \tau - x  }) [   (x_0 - \tau  )\underline{\tau}   -  \left\langle \underline{\tau} , \underline{x}  \right\rangle  ] }{   \|\tau  - x \|^6}   \bigg],\\
		C (x,\tau)  = &    \frac{1}{\pi^2   }  \bigg[     \frac{   {\tau}  + 3 \overline{x }   - 4 \overline{{\tau} }      }{   \|x - \tau \|^4  }     + \frac{  4  [   (x_0- \tau  )\underline{\tau}   -  \left\langle \underline{\tau} , \underline{x} \right\rangle  ] (   \overline{     \tau   - x  }) }{   \| \tau - x  \|^6}   \bigg]  \frac{\underline{x} }{   \| \underline{x} \|^2 },\\
	\end{align*}  
where $\tau , x  \in \mathbb H\setminus\mathbb R$ such that  $x = x_0   + \underline{x} $ and   $x \neq \tau $.  In addition,  for any  $x\in\mathbb H\setminus\mathbb R$ denote
	$$\displaystyle  \nu _x :=   2 d\hat{x}_0+ 2  \frac{\underline{x} }{ \|\underline{x}  \|^2 } \sum_{k=1}^3 x_k d \hat{x}_k.$$ 
\end{defn}

\begin{thm}\cite{GG1}\label{quaternionic_Borel-Pompeiu_form} (The global  quaternionic Borel-Pompeiu-type formula).
	Let $\Omega\subset\mathbb H$ be a domain such that  $\partial \Omega$ is a 3-dimensional compact smooth surface and $\overline{\Omega}\subset \mathbb H\setminus \mathbb R$. Then
	\begin{align*}   
		& \displaystyle \int_{\partial \Omega}  \|\underline{\tau} \|^2  \big[  {g(\tau)} \nu_\tau   A( x, \tau)  - A (\tau, x)  \nu_\tau  f(\tau) \big]   +  		\int_{\Omega} \big[ B (y, x)f(y)  -     
		g(y) C (x,y)	\big]
		dy \nonumber   \\ 
		&  +  2	\int_{\Omega}  \big[ 	   A (y, x)   G[f](y)    -    G_r[g](y) A (x,y)  \big]
		dy =       \left\{ \begin{array}{ll} f(x) + g(x), &  x\in \Omega, \\ 0 ,&  x\in \mathbb H\setminus\overline{\Omega},  \end{array}       \right.    \end{align*}
	for all  $f,g \in C^{1}(\overline{\Omega},\mathbb H)$.	
\end{thm}
 
\begin{prop} \cite{GG1}  \label{quaternionic_Cauchy_integral_form} (The  global quaternionic Cauchy-type integral formula).
	Let $\Omega\subset\mathbb H$ be a domain such that  $\partial \Omega$ is a 3-dimensional compact smooth surface and $\overline{\Omega}\subset \mathbb H\setminus \mathbb R$. Then
	\begin{align*} 
		&  \displaystyle \int_{\partial \Omega}   \|\underline{\tau} \|^2\big[  {g(\tau)} \nu_\tau   A ( x, \tau)  - A (\tau, x)   \nu_\tau   f(\tau) \big]    +  		\int_{\Omega} \big[ B (y, x)f(y)  -     
		g(y) C (x,y)	  \big]
		dy \nonumber \\
		&    =  \left\{ \begin{array}{ll} f(x) + g(x), &  x\in \Omega, \\ 0 ,&  x\in \mathbb H\setminus\overline{\Omega},  \end{array}       \right.    \end{align*}
	for all  $f\in \textrm{Ker}( G) \cap  C^{1}(\overline{\Omega},\mathbb H)$ and $g\in \textrm{Ker}(G_r)  \cap  C^{1}(\overline{\Omega},\mathbb H)$.
\end{prop}

\begin{prop} \cite{G} \label{StokesG} {(Integral  Stokes' theorem, $   G$  and $  G_r$).}\label{prop3.8}
	Let $\Omega\subset\mathbb H\setminus \mathbb R$ be a    domain such that  $\partial \Omega \subset \mathbb H\setminus \mathbb R$ is a 3-dimensional compact smooth surface. Then 
	$$  \displaystyle  \int_{\partial \Omega} g\nu_x  f = 4 \int_{\Omega} g  \frac{\underline{x} }{\| \underline{x} \|^2} f dx + 2 \int_{\Omega}  \frac{1}{\|\underline{x} \|^2}   \left(    G_r[g] f  +    g   G[f]  \right)dx  $$
	for all $f,g\in C^1(\overline{\Omega},\mathbb H)$.
\end{prop}

\begin{defn}\label{Moebustrans}The quaternionic M\"obius transformations are represented   by 
\begin{align*} \mathcal   T(q)= (aq+b)(cq+d)^{-1}, \end{align*}  where $a,b,c,d\in \mathbb H$ such that   $b-ac^{-1}d \neq 0 $ if  $c\neq 0$ and $ad^{-1}\neq 0$ if $c=0$.  Recall that    any  quaternionic M\"obius transformation is a composition of a finite number of the following basic transformations: 
\begin{enumerate}
	\item Rotation,  
	${}_{u}R_v(x)=uxv$ for all $x\in \mathbb H$ where $u,v\in\mathbb S^3$.
	\item The  quaternionic inversion, $J(x)=x^{-1}=\displaystyle  \frac{\bar x}{\|x\|^2} $ for all $x\in \mathbb H\setminus\{0\}$.
	\item Translation,    $T_b(x) = x + b $ for all $x\in \mathbb H$ where $b \in \mathbb H  $.
	\item Dilations, $D_{\lambda}(x) = \lambda  x$ for all $x\in\mathbb H$ with $\lambda  >0$.
\end{enumerate}
Particularly, the quaternionic rotations preserving $\mathbb R^3$ are given by 
${}_{u}R_{\bar u} (x)=ux\bar u$ for all $x\in \mathbb H$ where $u\in\mathbb S^3$.
\end{defn}
\begin{defn}    {Given two domains $\Xi,\Omega\subset  \mathbb R^4$, let
		$\alpha:\Xi \to \Omega$ be a one-to-one correspondence;  if $f$
		belongs to a function space on $\Omega$, then  denote  $W_\alpha: f
		\mapsto f\circ
		\alpha$.}
	{If $\beta$ is a $\mathbb
		H$-valued function then  denote   $  \  {{}^\beta M}: f \mapsto  \beta f,\quad M^\beta : f
		\mapsto   f \beta$ on the same function space}
\end{defn}   

\begin{prop}\cite{GG2}\label{ConfCovPropG} (Conformal covariance property of $G$) 
Let $\Omega\subset\mathbb H$ be a domain and  
	\begin{align*} 
		{\mathcal T}(x)= (ax+b)(cx+d)^{-1} ,
	\end{align*}  
	where $a,b,c,d\in\mathbb H$ satisfy   {that}  $a\bar c, (b-ac^{-1} d)\bar c, d \overline{(b-ac^{-1}d)} \in \mathbb R$. 
	Define 
	\begin{align} \label{functions}
		& \mathcal A_{\mathcal T}(x)=  \bar c, \quad  
		\mathcal B_{\mathcal T}(y) =     \|c\|  \|b-ac^{-1}d\| \bar c  (y - ac^{-1})^{-2},
	\end{align}
	where $y = {\mathcal T}(x)$.  
	Then  
	\begin{align*}
		G\circ ({}^{\mathcal A_{\mathcal T}}M\circ W_{\mathcal T})= ({}^{\mathcal B_{\mathcal T}}M\circ W_{\mathcal T})\circ G
	\end{align*}
	on  $C^1({\mathcal T} (\Omega),\mathbb H)$.
\end{prop}

 \section{Main results}\label{MR}

\subsection{Cliffordian framework}\label{CLF}

This section  introduces an extension of the function theory induced by $G$  in Clifford analysis   framework and its consequence for the $(n+1)$-dimensional material derivative.

\begin{defn}\label{Def111}
	Let $U,V\subset \mathbb R^{n }$  be two domains and 
	let  $(s,t) \subset \mathbb R$ be an open interval.  Given    
	$b \in C^1((s,t)\times V, \mathbb R_{0,n})$ and  a diffeomorphism  $\underline{a}=\sum_{i=1}^n a_i e_i  \in C^1(   U,  V)$, where $a_i$ is a real-valued function for all $i$. Define  $a : (s,t)\times U \to (s,t)\times V$  given by  $$a(x) =x_0 + \sum_{i=1}^n a_ i (\underline{x}) e_i. $$
	Denote 	$y= a(x)$. Then    $x_0 =(a(x))_0 = y_0$,  $\underline{y}=    \underline{ a}( \underline{ x})  $ for all $x\in (s,t)\times U$, and  
		$$a^{-1}(y) =y_0 +  \sum_{j=1}^n (a^{-1}) _ j(\underline{y}) e_j, $$ where 
		$(a^{-1}) _ j$ is the j-th real-component of $(\underline{a})^{-1}$ for $j=1,\dots ,n$. 
Given  $f\in C^{1} ((x,t)\times U,\mathbb R_{0,n} )$ define:
 \begin{align*}
 H_{a,b}[f] (x):= & ( G [b  ]\circ a )(x) f (x) +     |\underline{a }(x)|^2 \ ( b\circ a ) (x) \frac{\partial  f }{\partial x_0} (x)  \\
   &  + \underline{a } (x)\
	( b \circ a ) (x) \sum_{j=1}^n  \left(  \sum_{i=1}^n  a _i (x) \   \frac{\partial   ( a^{-1})_j  }{\partial  y_i} \circ a  (x) \right) \frac{\partial  f  }{\partial  x_j}(x) , \quad \forall x\in (s,t)\times U.  
	\end{align*}	 \end{defn}

From now on we shall assume the hypotheses and notations of functions $a,b$ and their domains    introduced in the previous definition throughout this subsection.
 Also given the functions $f$,  $\alpha$ and $\beta$ we shall denote  $W_\alpha[ f]= f\circ
		\alpha$ and $  {{}^\beta M}[ f ]= \beta f $ if the composition and product are valid, respectively.

\begin{prop}\label{HyG} 
$H_{a,b} =W_{a } \circ G \circ {}^{b}M\circ W_{a^{-1}}  $ on  $C^{1} ( (s,t)\times U, \mathbb R_{0,n})$
 or equivalently,  
 $$H_{a,b}[f]=W_{a } \circ G \circ {}^{b}M\circ W_{a^{-1}}  [f]  , 
 \quad \forall  f\in C^{1} ((x,t)\times U,\mathbb R_{0,n}).$$
\end{prop}
\begin{proof} From direct computations we see that
	\begin{align*}
		&  W_{a } \circ G \circ {}^{b}M \circ W_{a^{-1}}  [f] =  W_{a }( G ( b f \circ {a^{-1}} ) )
		\\
		=& W_{a }( G ( b ) f \circ {a^{-1}}   )   +    
		W_{a } \left( |\underline{y}|^2  b  \frac{\partial  f \circ a^{-1} }{\partial y_0}  + \underline{y} 
		b \sum_{i=1}^n   y_i  \frac{\partial  f \circ a^{-1}  }{\partial  y_i}  \right) \\
		= & ( G [b  ]\circ a ) f +     |\underline{a }|^2 ( b\circ a )  \frac{\partial  f }{\partial x_0}    + \underline{a } 
		( b \circ a ) \sum_{j=1}^n  \left(  \sum_{i=1}^n  a _i    \frac{\partial   ( a^{-1})_j  }{\partial  y_i} \circ a \right) \frac{\partial  f  }{\partial  x_j}    
	\end{align*}   	
\end{proof}

\begin{cor}\label{corKerHKerG}
$f\in \textrm{Ker}(H_{a,b})$ if and only if \ ${}^{b}M \circ W_{a^{-1}}  [f]\in\textrm{Ker}(G)$.
\end{cor}

\begin{rem}
	Although the operator $H_{a,b}$ is obtained from $G$ as shown in Proposition \ref{HyG} and 
	  we can surely see that many properties of  $H_{a,b}$   are consequences of the developed theory of $G$, we must note that  $H_{a,b}$ is a  generalization of   $G$   since if $U=V$, $a$ is the identity function $I$ and $b$ is the  constant function $1$ then we obtain  that $H_{I,1} = G$
	 on  $C^{1} ( (s,t)\times U, \mathbb R_{0,n})$. Therefore, the study of the operator $H_{a,b}$ offers an extension of the theory on the operator $G$ and 
	 all results presented in Section \ref{prel} can be used to show many properties of operator $H_{a,b}$ using  Proposition \ref{HyG} and the operator $G$ with direct computations extending the function theories  	   $\mathcal{SM}((s,t)\times U)$, $\mathcal { HM}((s,t)\times U)$ and  $\mathcal {GM}((s,t)\times U)$.

	On the other hand, it is important to comment that the  part    
	$$ \sum_{j=1}^n  \left(  \sum_{i=1}^n  a _i    \frac{\partial   ( a^{-1})_j  }{\partial  y_i} \circ a \right) \frac{\partial  f  }{\partial  x_j},$$
of   $H_{a,b}[f]$  	is a particular case  of the operator: 
	$$ \sum_{j=1}^n   u_j \frac{\partial  f  }{\partial  x_j},$$
when $ u_j=    \sum_{i=1}^n  a _i    \dfrac{\partial   ( a^{-1})_j  }{\partial  y_i} \circ a  $ for $j=1,\dots, n$,  which is an important part  of  the multidimensional material derivative.
\end{rem}

As our objective is to show new  properties of $D_u$ we are going to assume   that $b=1$. 

\begin{defn} Denote  $\mathcal H_a = H_{a,1}$ on  $C^{1} ( (s,t)\times U, \mathbb R_{0,n})$, i.e., 
	\begin{align*} \mathcal H_a (f) (x) := {\underline a}   ( {x})  \frac{\partial f}{\partial x_0} ( x)  
		- \sum_{i=1}^n     \left( \sum_{j=1}^n   a _j  ( {x}) \frac{\partial (a^{-1})_i  }{\partial  y_j}\circ    a  ( {x})  \right)  \frac{\partial f}{\partial  x_i} (x),  
	\end{align*}
	for all $x\in (s,t)\times U$ and all $f\in C^{1} ((s,t)\times U, \mathbb R_{0,n})$.
\end{defn}

\begin{rem}\label{remDyH}	
From Proposition \ref{HyG} 	one has that 
	\begin{align}\label{GyH}
		W_{a } \circ G  \circ W_{a^{-1}}  [f] = G[f\circ a^{-1}]\circ a  = -\underline{a} \mathcal H_a (f), \quad  \forall  f\in C^{1}((s,t)\times U , \mathbb R_{0,n})
	\end{align}
and supposing that  
$$u_i =   \sum_{j=1}^n a_j  \frac{\partial (a^{-1})_i  }{\partial  y_j}\circ    a    ,$$ 
for $i=1,\dots, n$, then $\mathcal H_a$ has a  direct participation in the multidimensional  incompressible Navier-Stokes equations.   What is more,  the  multidimensional material derivative satisfies 
	\begin{align*} {\mathrm D}_u (f) = &  \frac{\partial  f}{\partial x_0 } 
		+ \sum_{j=1}^n u_j \frac{\partial f }{\partial x_j } 
		=  (1  +{\underline a} ) 
		{\frac {\partial f}{\partial x_0}}
		- \mathcal H_a(f)  , \quad \forall f\in C^1((s,t)\times U, R^n)    
	\end{align*}
	and the $n-$dimensional incompressible Navier-Stokes equations are given by
	\begin{align*}
		(1  +{\underline a} )   \frac{\partial u}{\partial x_0}    
		- \mathcal H_a(u)  - \alpha \sum_{j=1}^n   \frac{\partial^2 u}{\partial x_j^2 }  =-  \nabla p+ f(x, t).\end{align*}
\end{rem} 

\begin{ex}\label{examples}
We will see some examples of   $\mathcal H_a$   in terms of  simple functions $a$ but that provide us with interesting generalizations of $G$  and  all its consequences in certain function theories.
 Assume that   $a_j(x)= a_j(x_j) = y_j$ for $j=1,\dots, n$, then   
	\begin{align*} \mathcal H_a (f) (x) = {\underline a}  (\underline{x})  \frac{\partial f}{\partial x_0} ( x)  
		- \sum_{i=1}^n     \frac{ a _j  (\underline{x})}{a'_j  (\underline{x})}   \frac{\partial f}{\partial  x_i} (x) , 
	\end{align*}
	for all $f\in C^1((s,t)\times U,\mathbb R_{0,n})$. In the following examples we consider particular cases of domains $(s,t)\times U$ in which we admit $(s,t)=(-\infty, \infty)$, for short explanations. 
	\begin{enumerate}
		\item If  $\displaystyle  a(x)= x_0+  \sum_{j=1}^n(r_j{x_j}+ s_j) e_j$ for all $x_0,y_0, y_j,x_j\in \mathbb R$,  for $j=1,\dots, n$, where $r_j\neq 0 $ for $j=1,\dots, n$, then  
		\begin{align*} \mathcal  H_a (f) (x) = \sum_{j=1}^n (r_j{x_j}+ s_j) e_j \frac{\partial f}{\partial x_0} ( x)  
			-   \sum_{i=1}^n     ( x_i + \frac{s_i}{r_i})    \frac{\partial f}{\partial  x_i} (x). 
		\end{align*} 
		
		\item Given $\alpha_1,\dots,\alpha_n\in \mathbb R\setminus \{0\}$ and    $ \displaystyle  a(x)= x_0+  \sum_{j=1}^n{x_j}^{\alpha_j} e_j$, 
		where $x_0,y_0 \in \mathbb R$,  $y_j,x_j >0$  for $j=1,\dots, n$.
		Then  
		\begin{align*}  \mathcal  H_a (f) (x) = \left(   \sum_{j=1}^n{x_j}^{\alpha_j} e_j \right)  \frac{\partial f}{\partial x_0} ( x)  
			-   \sum_{i=1}^n     \frac{x_i}{ \alpha_i}    \frac{\partial f}{\partial  x_i} (x). 
		\end{align*}
				
		\item Set  $\displaystyle  a(x)= x_0+ \sum_{j=1}^n e^{x_j} e_j$,  where $x_0, y_0\in \mathbb R$ and  $y_j,x_j \in (0,\infty)$,  for $j=1,\dots, n$. Therefore,
		
		\begin{align*}  \mathcal  H_a (f) (x) = \sum_{j=1}^n e^{x_j} e_j   \frac{\partial f}{\partial x_0} ( x)  
			-  \sum_{i=1}^n         \frac{\partial f}{\partial  x_i} (x) .
		\end{align*}
		
		\item Set  $ \displaystyle  a(x)= x_0+ \sum_{j=1}^n \sin({x_j}) e_j$,  
		where  $x_0,y_0, y_j,x_j\in (0, \pi)$,   for $j=1,\dots, n$. Then
		
		\begin{align*}  \mathcal  H_a (f) (x) = \sum_{j=1}^n \sin({x_j}) e_j  \frac{\partial f}{\partial x_0} ( x)  
			-  \sum_{i=1}^n   \tan(x_i)      \frac{\partial f}{\partial  x_i} (x). 
		\end{align*}
		
		\item Set  $\displaystyle  a(x)= x_0+ \sum_{j=1}^n \ln({x_j}) e_j $,  
		where $x_0, y_0\in \mathbb R$ and  $y_j,x_j \in (0,\infty)$,  for $j=1,\dots, n$. Therefore,
		
		\begin{align*}  \mathcal  H_a (f) (x) = & \sum_{j=1}^n \ln({x_j}) e_j   \frac{\partial f}{\partial x_0} ( x)  
			-  \sum_{i=1}^n    x_i \ln(x_i)    \frac{\partial f}{\partial  x_i} (x) 
			\\
			= & \sum_{j=1}^n \ln({x_j}) \left(e_j   \frac{\partial f}{\partial x_0} ( x)  
			-    x_j   \frac{\partial f}{\partial  x_j} (x) \right).\end{align*}
	\end{enumerate}
\end{ex}
\begin{prop} \label{propSMHMH} Let  $a((s,t)\times U )\subset \mathbb R^{n+1}$ be  an axially symmetric s-domain and given $f\in C^1((s,t)\times U ),  \mathbb R_{0,n} )$.
	\begin{enumerate}
		\item  If  $
		f\circ  a^{-1}   \in \mathcal { SM}( (s,t)\times V))$  then   $ \mathcal  H_a[f]=0$ on $(s,t)\times U$.
		\item  If $f\circ  a^{-1} \in \mathcal{ HM}((s,t)\times V)$  then  $ \mathcal  H_a[f]=0$ on $(s,t)\times  U$.
	\end{enumerate}
\end{prop}
\begin{proof} If 	$f\circ  a^{-1}   \in \mathcal { SM}( (s,t)\times V )$  or $f\circ  a^{-1} \in \mathcal{ HM}((s,t)\times V)$ then  Theorem \ref{teoHGM}   
 shows us that  $f\circ  a^{-1}    \in\mathcal{ GM} ( (s,t)\times V)$.  From  \eqref{GyH} we see that 
 $0= G[f\circ a^{-1}]\circ a  = -\underline{a} \mathcal H_a (f)$ on  $(s,t)\times U$. Recall that  $a$ is a diffeomorphism. Therefore   $ \mathcal H_a (f)=0$ on  $(s,t)\times U$.
\end{proof}

\begin{prop}\label{kerHPower}(On $\textrm{Ker}(\mathcal H_a)$ and the power series development). 
 Given   $f\in C^1((s,t)\times U ,  \mathbb R_{0,n} )$ such that 
   $f\circ a^{-1}\in \mathcal {SM} ((s,t)\times V) $  and $\mathbb B^{n+1}(0,1) \subset (s,t)\times V$.  Then  $f\in \textrm{Ker}(\mathcal H_a)\cap C^1((s,t)\times U ,\mathbb R_{0,n}) $ and there exists a sequence $(a_n)_{n\geq 0}$ of elements of $\mathbb R_{0,n}$ such that 
	\begin{align*}  f (x)=   \sum_{n=0}^\infty (a(x))^n a_n,\quad \forall x\in a^{-1}(\mathbb B^{n+1}(0,1)).  
	\end{align*}
\end{prop}
\begin{proof}	As $f\circ a^{-1}\mid_{\mathbb B^{n+1}(0,1)}  \in \textrm{Ker}(   G)
	\cap C^1(\mathbb B^{n+1}(0,1),R^n) $. Then  
	from  \eqref{GyH} we have   $f\in \textrm{Ker}(\mathcal H_a)\cap C^1((s,t)\times U ,\mathbb R_{0,n}) $ and Proposition \ref{Power} implies that 
	 	\begin{align*}  f\circ a^{-1}\mid_{\mathbb B^{n+1}(0,1)}  (y) = \sum_{n=0}^\infty y^n a_n,\quad \forall y\in\mathbb B^{n+1}(0,1).  
	\end{align*}
	Therefore using $y=a(x)$ we obtain the power series.
\end{proof}
 
\begin{thm}\label{KerHRepr}
	(Representation formula). 
		Suppose that  $ (s,t)\times V  $ is an axially symmetric
	s-domain and  given   $f\in C^1((s,t)\times U ,  \mathbb R_{0,n} )$ such that 
  and   $f\circ a^{-1}\in \mathcal {SM} ((s,t)\times V) $. Then 
	$f  \in \textrm{Ker}(\mathcal H_a)$ and  for any $x\in  (s,t)\times U $ we have that  
	$$
	f   (x) = \frac{1}{2} (1 - I_y I)f (x_I^+) + \frac{1}{2}  (1 + I_y I) f( x_I^- ), \quad I\in \mathbb  S,
	$$where 
	$ a (x) = y =  u + I_y v $,  $ x_I^+=  a^{-1} (u + I_y  v)$ and  $ x_I^-=  a^{-1}(u - I_y  v)$.
\end{thm}
\begin{proof}
  Theorem \ref{Repth} and \eqref{GyH} give us that $f  \in \textrm{Ker}(\mathcal H_a)$ and  if  $I\in \mathbb  S$ then 
	$$
	f \circ a^{-1}  (y) = \frac{1}{2} (1 - I_y I)f\circ a^{-1} (u + I  v) + \frac{1}{2}  (1 + I_y I) f( u - I v ),  
	$$ where $y= u + I_y v$. The main identity is obtained using 
	$y= a(x) $,  $ x_I^+=  a^{-1} (u + I  v)$ and  $ x_I^-=  a^{-1}(u - I  v)$.

\end{proof}

\begin{lem}\label{kerHSplit} (Splitting Lemma). 
	Suppose that  $ (s,t)\times V  $ is an axially symmetric
	s-domain and  given $f\in C^1((s,t)\times U ,  \mathbb R_{0,n} )$ such that 
	 $f\circ a^{-1}\in \mathcal {SM} ((s,t)\times V) $. Then   
	$f  \in \textrm{Ker}(\mathcal H_a)$ and for every $I = I_1 \in \mathbb S$, let $ I_2, 
	\dots , I_n$ be
	a completion to a basis of $\mathbb R_{0,n}$ satisfying the defining relations $I_rI_s + I_sI_r 
	= -2\delta_{r,s}$. Then there exist $2^{n-1}$ holomorphic functions $F_A :( (s,t)\times V  )_{I}= ( (s,t)\times V  ) \cap \mathbb C_I\to \mathbb C_I$ such that
	for every $x\in a^{-1}( ( (s,t)\times V  )_{I} )$  we have 
	$$f(x) = \sum^{n-1}_{|A|=0}
	F_A(a(x))I_A, \quad  I_A = I_{i_1}\cdots I_{i_s},$$
where $A = i_1\dots i_s$ is a subset of $\{2, . . . , n\}$, with $i_1 < \dots < i_s$, or $|A| = 0$,  $I_{\emptyset} = 1$.
\end{lem}
\begin{proof}
	 Lemma \ref{Split}  and \eqref{GyH} give us that  
	 	 $f  \in \textrm{Ker}(\mathcal H_a)$ and  
	$$f\circ a^{-1}(y) = \sum^{n-1}_{|A|=0}
	F_A(y)I_A, \quad  I_A = I_{i_1}\cdots I_{i_s},$$
for all $y\in  ( (s,t)\times V  )_{I}  $. The formula is obtained doing 
$y=a(x)$.
	 
\end{proof}

\begin{thm} \label{KerHCauchyForm} (Cauchy formula). Let $ (s,t)\times V \subset \mathbb R^{n+1}$ be an open set and   given $f\in C^1((s,t)\times U ,  \mathbb R_{0,n} )$ such that 
		$f\circ  a^{-1} \in \mathcal {HM}( (s,t)\times V)$ or 
	$f\circ  a^{-1} \in \mathcal{GM}_p((s,t)\times V)
	$. Let $\Omega \subset (s,t)\times V$ be a bounded axially symmetric open set. Suppose
	that the boundary of $\Omega\cap \mathbb  C_I$ consists of a finite number of rectifiable Jordan curves
	for any $I\in \mathbb  S$. Then  
	$f  \in \textrm{Ker}(\mathcal H_a)$  and 
	if $y \in (s,t)\times U$, we have
	$$
	f (x) = \frac{1}{2\pi} \int_{\partial (\Omega\cap \mathbb C_I )}
	S^{-1}(s, a(x))ds_I (f\circ  a^{-1})(s), $$
	where $ds_I = \frac{ds}{I}$ and
	$$
	S^{-1}(s, a(x)) = -\left[ (a(x))^2-  2a(x) \textrm{Re}[s] + 
	|s|^2\right]^{-1}(a(x) - \overline{s}),$$
	and the integral does not depend on $\Omega$ nor on the imaginary unit $I \in \mathbb S$.
\end{thm}
\begin{proof} Using 
 Theorems \ref{CauchyFormSM} and \ref{CauchyFormSM2} and \eqref{GyH} one has that 
 $f  \in \textrm{Ker}(\mathcal H_a)$ and  
 	$$
	f \circ a^{-1}(y) = \frac{1}{2\pi} \int_{\partial (\Omega\cap \mathbb C_I )}
	S^{-1}(s, y)ds_I (f\circ  a^{-1})(s). $$
Use $y=a(x)$. 
\end{proof}

\begin{rem}
Denoting   $u_i =   \left( \sum_{j=1}^n   a _j   \frac{\partial (a^{-1})_i  }{\partial  y_j}\circ    a     \right) $ for $i=1,\dots, n$ we see that the multidimensional  material derivative  satisfies 
	$$ { {D}_u  }f =   (1  +{\underline a} ) 
	{\frac {\partial f }{\partial x_0}}
	- \mathcal H_a[f], \quad  \forall f\in C^1((s,t)\times U ,  \mathbb R_{0,n} ).$$
	Then,  $f\in\textrm{Ker}(\mathcal H_a)$ iff
	\begin{align}\label{DkerHa}
		{  D}_u (f) =    (1  +{\underline a} ) 
		{\dfrac {\partial f}{\partial x_0}},  \quad \textrm{on} \quad (s,t)\times U.
	\end{align}
	Therefore,      Propositions \ref{propSMHMH}, \ref{kerHPower},  Theorems  
	\ref{KerHRepr},  \ref{KerHCauchyForm} 
	and Lemma \ref{kerHSplit}  are 
	  are valid for the set of functions induced by $D_u$ under condition \eqref{DkerHa}.
	  
Also, note that    $f\in\textrm{Ker}(D_u)$ iff
	\begin{align}\label{HakerDu}
  \mathcal H_a [f ] = (1  +{\underline a} ) 
	{\frac {\partial f }{\partial x_0}}
	 , \quad \textrm{on} \quad (s,t)\times U,
	\end{align}
which  shows us that  
	  the function  set  $\textrm{Ker}(D_u)$  is directly associated to is associated with a problem of generalized functions of $\mathcal H_a $ in a Vekua type sense.

	 \end{rem} 

\begin{rem}
	Using the theory  of the  operator
	$$
	G_R[f](x) = |\underline{x}|^2 \frac{\partial f}{\partial x_0}(x)
	+ \sum_{j=1}^n x_j \frac{\partial f}{\partial x_j}(x)  \underline{x}  ,$$
	we can obtain some properties of the operator
	\begin{align*} \mathcal H_{a,R} (f) (x) :=  \frac{\partial f}{\partial x_0} ( x)\  {\underline a}   (\underline{x})   - \sum_{i=1}^n     \left( \sum_{j=1}^n   a _j  (\underline{x}) \frac{\partial (a^{-1})_i  }{\partial  y_j}\circ    a  (\underline{x})  \right)  \frac{\partial f}{\partial  x_i} (x),  
	\end{align*} for all $x\in (s,t)\times U$,
	and  if  $u_i =   \left( \sum_{j=1}^n   a _j  (\underline{x}) \frac{\partial (a^{-1})_i  }{\partial  y_j}\circ    a  (\underline{x})  \right) $ for $i=1,\dots, n$,  then 
	$$ {  {D}_u  (f)} =   
	{\frac {\partial f }{\partial y_0}}  \ (1  +{\underline a} )
	- \mathcal H_{a,R} [f].$$
The properties obtained on $ {D}_u$ will be very similar to those obtained in this subsection using $G$ and $\mathcal H_a$.\end{rem}

\subsection{Quaternionic framework}\label{QFR}

\begin{defn}	
	Let  $U,V\subset \mathbb R^{3}$ be  two domains and $(s,t) \subset \mathbb R$ such that   there exists   
	$b \in C^1((s,t)\times V, \mathbb H)$ and a  diffeomorphism  $\underline{a} 
	= \sum_{i=1}^3 a_ i   e_i   \in C^1(   U,  V)$, where $a_i$ is a real valued function for all $i$. Define  $a:  (s,t)\times U \to (s,t)\times V$ as follows  
		$$a(x) =x_0 +  \sum_{i=1}^3 a_ i (\underline{x}) e_i, \quad \forall x\in (s,t)\times U.$$ 
If we denote 	$y= a(x)$ then   $x_0 =(a(x))_0 = y_0$ and $\underline{y}= \underline{  a}( \underline x)$.  Also 	$x= a^{-1}(y) =y_0 +  \sum_{j=1}^3 (a^{-1}) _ j(y) e_j$, where 
$(a^{-1}) _ j$ is  the j th real component function of $a^{-1}$. 

Define  
	\begin{align*}
		& H_{a,b}[f] :=\\
		&  (  G [b  ]\circ a ) f +     |\underline{a }|^2 \ ( b\circ a )  \frac{\partial  f }{\partial x_0}    + \underline{a } \   ( b \circ a ) \sum_{j=1}^3  \left(  \sum_{i=1}^3 a _i  \  \frac{\partial   ( a^{-1})_j  }{\partial  y_i} \circ a \right) \frac{\partial  f  }{\partial  x_j} , \\   
		& H_{a,b,r}[f] := \\
		& f  G_r [b  ]\circ a  +      \frac{\partial  f }{\partial x_0} |\underline{a }|^2  \  ( b\circ a )    + \sum_{j=1}^3  \left(  \sum_{i=1}^3 a _i   \ \frac{\partial   ( a^{-1})_j  }{\partial  y_i} \circ a \right) \frac{\partial  f  }{\partial  x_j}  
		( b \circ a )    \underline{a } , \\
		&\mathcal H_a (f) : =  {\underline a}     \frac{\partial f}{\partial x_0}   
		- \sum_{i=1}^3     \left( \sum_{j=1}^3   a _j   \ \frac{\partial (a^{-1})_i  }{\partial  y_j}\circ      a    \right)  \frac{\partial f}{\partial  x_i}  ,\\
		&\mathcal  H_{a,r} (f) :=  \frac{\partial f}{\partial x_0}   {\underline a}   
		- \sum_{i=1}^3     \left( \sum_{j=1}^3   a _j   \ \frac{\partial (a^{-1})_i  }{\partial  y_j}\circ     a    \right)  \frac{\partial f}{\partial  x_i}  ,
	\end{align*}
	for all $f\in C^{1} ((s,t)\times U,\mathbb H)$. 
\end{defn}

From now on we shall assume the hypotheses and notations of functions $a,b$ and their domains    introduced in the previous definition throughout this subsection. Also given the functions $f$,  $\alpha$ and $\beta$   denote  $W_\alpha[ f]= f\circ
		\alpha$ and $  {{}^\beta M}[ f ]= \beta f $ if the composition and product are valid, respectively.

\begin{prop}\label{HGQuaternions}  Given 
	$f\in C^{1}((s,t)\times U  , \mathbb H)$ 
	we have that    
	\begin{align*}
		H_{a,b}[f]=& W_{a } \circ  G \circ {}^{b}M \circ W_{a^{-1}}  [f]  ,\\
		   {}^{-\underline{a}}M\circ \mathcal H_a (f) = & W_{a } \circ  G   \circ W_{a^{-1}}  [f], \\
		{}^{-\underline{a}}M\circ \mathcal H_{a,r} (f) = &  W_{a } \circ  G_r   \circ W_{a^{-1}}  [f], \quad \textrm{ on} \ \   (s,t)\times U.
	\end{align*}
	
\end{prop}
\begin{proof}
	Follows from similar computations to those presented in Proposition \ref{HyG} and Corollary \ref{corKerHKerG}.
\end{proof}

\begin{rem} Given 	$f\in C^{1}((s,t)\times U  , \mathbb H)$. 
Note that $f\in \textrm{Ker}(\mathcal H_{a}) $ iff \ $    W_{a^{-1}}  [f]\in\textrm{Ker}(G)$.
	In addition,  $f\in \textrm{Ker}(\mathcal H_{a,r}) $ iff 
		 $  W_{a^{-1}}  [f]\in\textrm{Ker}(G_r)$.
\end{rem}

\begin{ex}
The quaternionic right version of operators given in Examples \ref{examples} are the following:
	\begin{align*}
	\mathcal 	H_{a,r} (f) (x) =& \frac{\partial f}{\partial x_0} ( x)  \sum_{j=1}^3(r_j{x_j}+ s_j) e_j  
		-   \sum_{i=1}^3     ( x_i + \frac{s_i}{r_i})    \frac{\partial f}{\partial  x_i} (x), \\
	\mathcal 	H_{a,r} (f) (x) =&  \frac{\partial f}{\partial x_0} ( x) \left(   \sum_{j=1}^3{x_j}^{2k} e_j \right)   
		- \frac{1}{2k}  \sum_{i=1}^3     x_i    \frac{\partial f}{\partial  x_i} (x), \\
	\mathcal 	H_{a,r} (f) (x) = &   \frac{\partial f}{\partial x_0} ( x) \sum_{j=1}^3 e^{x_j} e_j   
		-  \sum_{i=1}^3         \frac{\partial f}{\partial  x_i} (x) ,\\
	\mathcal 	H_{a,r} (f) (x) = &   \frac{\partial f}{\partial x_0} ( x) \sum_{j=1}^3 \sin({x_j}) e_j  
		-  \sum_{i=1}^3   \tan(x_i)      \frac{\partial f}{\partial  x_i} (x),\\
	\mathcal 	H_{a,r} (f) (x) = &   \frac{\partial f}{\partial x_0} ( x)  \sum_{j=1}^3 \ln({x_j}) e_j 
		-  \sum_{i=1}^3   x_i \ln(x_i)    \frac{\partial f}{\partial  x_i} (x),
	\end{align*}
respectively, where $f$ satisfies the respective necessary conditions for each operator. 
\end{ex}

\begin{ex} An interesting example in the    quaternionic case occurs when   $a$ is the quaterninic rotation preserving $\mathbb R^3$, i.e.,  there exists
	$c\in \mathbb S^2$ such that   $a(x)=c x\bar c$ for all $x\in \mathbb H $. 
	Then  direct computations allows to see that
	\begin{align*}   \mathcal  H_a (f) (x) = &  c\underline{x}\bar c  \frac{\partial f}{\partial x_0} ( x)  
		- \sum_{i=1}^3     \left( \sum_{j=1}^3 (c\underline{x}\bar c e_j + e_j c\underline{x}\bar c )   (\bar c e_j  c e_i + e_i  \bar ce_j c )   \right)   \frac{\partial f}{\partial  x_i} (x) ,\\ 
    \mathcal   H_{a,r} (f) (x) =  
		&  \frac{\partial f}{\partial x_0} ( x)   c\underline{x}\bar c 
		- \sum_{i=1}^3     \left( \sum_{j=1}^3 (c\underline{x}\bar c e_j + e_j c\underline{x}\bar c )   (\bar c e_j  c e_i + e_i  \bar ce_j c )   \right)   \frac{\partial f}{\partial  x_i} (x) , \quad \forall x\in \mathbb H, 
	\end{align*}
	for all $f\in C^1(\mathbb H , \mathbb H)$.
\end{ex}

\begin{rem}\label{remDHQuater}
	Suppose that  $u_i =   \sum_{j=1}^3   a _j   \frac{\partial (a^{-1})_i  }{\partial  y_j}\circ    a      $ for $i=1,2, 3$. The 4-dimensional incompressible Navier-Stokes equation and the   material derivative  
are rewritten as
	\begin{align*}
		(1  +{\underline a} )   \frac{\partial u}{\partial x_0} &   
		- \mathcal H_a(u)  - \alpha \sum_{j=1}^3  \frac{\partial^2 u}{\partial x_j^2 }  =-  \nabla p+ f(x, t),\\
	  {D}_u f = &   (1  +{\underline a} ) 
		{\frac {\partial f}{\partial y_0}}
		- \mathcal  H_a(f)  , 
	\end{align*} 
	respectively, 	and  Euler equations of gas-dynamics,
	\begin{align*}
		\frac{\partial \rho }{\partial t}   + \nabla\cdot  (\rho u) =& 0 ,\\
		\rho\left(  \frac{\partial u }{\partial t}   + (u \cdot \nabla)   u \right)+ \nabla p  =& \rho a ,\\
		\frac{\partial p  }{\partial t}   + (u \cdot \nabla)   p + \gamma_p ( \nabla \cdot u ) =& 0 ,\\
	\end{align*}
	becomes at 
	\begin{align*}
		\frac{\partial \rho }{\partial y_0}   + \nabla\cdot  (\rho u) =& 0 ,\\
		\rho\left(  (1+{\underline a} ) 
		{\frac {\partial u}{\partial y_0}}
		- \mathcal H_a(u)  \right)+ \nabla p  =& \rho a ,\\
		(1+  {\underline a} ) 
		{\frac {\partial p}{\partial y_0}}
		-\mathcal H_a(p)  + \gamma_p ( \nabla \cdot u ) =& 0.\\
	\end{align*}
	In a similar way the operator $\mathcal H_{a,r}$ can be used to express the 4-dimensional incompressible Navier-Stokes equation,  the   material derivative  	and  Euler equations of gas-dynamics. 
\end{rem}

\begin{prop}\label{SerKerHa} (Power series development).  
	Suppose     $\mathbb B^{4}(0,1)\subset (s,t)\times V $. If $f\in \textrm{Ker}( \mathcal H_a)\cap C^1((s,t)\times U,\mathbb H)$  then there exists a sequence of quaternions $(a_n)_{n\geq 0}$ such that    
	\begin{align*}  f (x)=   \sum_{n=0}^\infty (a(x))^n a_n,\quad \forall x\in a^{-1}(\mathbb B^{4}(0,1)).  
	\end{align*}
	Similarly, if $f\in \textrm{Ker}( \mathcal H_{a,r})\cap C^1((s,t)\times U,\mathbb H) $  then    
	\begin{align*}  f (x)=   \sum_{n=0}^\infty  a_n (a(x))^n ,\quad \forall x\in a^{-1}(\mathbb B^{4}(0,1)).  
	\end{align*}
\end{prop}
\begin{proof}
We have $f\circ a^{-1} \in \textrm{Ker}(G)	\cap C^1(\mathbb B^{4}(0,1),\mathbb H)$. From the power series expansion theorem for slice regular  functions, see \cite{CSS},  there exists a sequence of quaternions $(a_n)_{n\geq 0}$   such that 
	\begin{align*}  f\circ a^{-1} (y) = \sum_{n=0}^\infty y^n a_n,\quad \forall y\in\mathbb B^{4}(0,1).  
	\end{align*}
	Notation $y=a(x)$ gives us the first power series development. 
		
	In addition,  if $f\in \textrm{Ker}(\mathcal  H_{a,r})\cap C^1((s,t)\times U,\mathbb H) $, then the power series expansion theorem for  $f\circ a^{-1}$ gives us the second result.   
\end{proof}

\begin{thm}\label{thmKerHASSD} 
Let  $(s,t)\times V \subset \mathbb H$ be  an axially symmetric s-domain and 
$f\in   C^1((s,t)\times U,\mathbb H) $.
\begin{enumerate}
			\item If  $f\circ a^{-1} \in \mathcal {SR}((s,t)\times V)$  then $f \in \textrm{Ker} (\mathcal H_a) $.
			\item If $f\circ a^{-1} \in  \mathcal {HR}( (s,t)\times V)$  then $f \in \textrm{Ker} (\mathcal H_a)  $.
\end{enumerate}  
\end{thm}
\begin{proof}Use Propositions \ref{HGquater} and \ref{HGQuaternions}. 
\end{proof}

The previous proposition has a quaternionic right version applying $G_r$ and $H_{a,r}$.

\begin{thm}\label{quaternionic_Borel-Pompeiu_form_La} (Borel-Pompeiu-type formula associated to $H_a$ and $H_{a,r}$).
	Let $\Omega\subset (s,t) \times U$ be a domain such that 
	such that $\overline{a(\Omega)}\subset (s,t)\times V  $ and  $\partial ( \Omega )$  
	and $\partial ( a(\Omega))$ are 3-dimensional compact smooth surfaces. Also suppose  $\overline{a(\Omega)}\subset \mathbb H\setminus \mathbb R$. 
	Then
	\begin{align*}   & \displaystyle \int_{\partial  \Omega }  \|\underline a(\underline \tau)  \|^2  \big[  {g (\tau)} J_{a}\nu_{a(\tau)}   A ( a(x), a(\tau))  - A (a(\tau), a(x)) J_a \nu_{a(\tau)}   f (\tau) \big]  \nonumber \\
		& +  		\int_{ \Omega } \big[ B (a(y), a(x))f (y)  -     
		g (y) C (a(x),a(y))	\big]
		M_a	dy  \nonumber   \\ 
		&  +  2	\int_{ \Omega }  \big[ 	   A (a(y), a(x)) 
		\vec{a}(\underline y) \mathcal H_a [f] (y)    + 
		\mathcal H_{a,r} (g)  \underline{a}(\vec  y)  A (a(x),a(y))  \big]
		M_a dy \nonumber \\
		= &       \left\{ \begin{array}{ll} f (x) + g (x), &  x\in  ( \Omega) , \\ 0 ,&  x\in \mathbb H\setminus\overline{  ( \Omega) },  \end{array}       \right.    
	\end{align*}
	for all  $f ,g  \in C^{1}(\overline{ \Omega },\mathbb H)$, where $J_a$ and $M_a$ are the differential increases in the  differential  volume caused by the variable change $x=a(y)$ in the respective integrals. For example, in $\Omega$ we see that  $M_a$ is given en terms of the determinant of the Jacobian matrix of $a$.	
\end{thm}

\begin{proof} Apply   Theorem \ref{quaternionic_Borel-Pompeiu_form} to functions  	   $f\circ a^{-1},g\circ a^{-1} \in C^{1}(\overline{a(\Omega)},\mathbb H)$    in  $a(\Omega)$ to see  that 
	\begin{align*}   & \displaystyle \int_{\partial a(\Omega)}  \|\underline{\tau} \|^2  \big[  {g\circ a^{-1}(\tau)} \nu_\tau  A ( x, \tau)  - A (\tau, x)  \nu_\tau   f\circ a^{-1}(\tau) \big]  \nonumber \\
		& +  		\int_{a(\Omega)} \big[ B (y, x)f\circ a^{-1}(y)  -     
		g\circ a^{-1}(y) C(x,y)	\big]
		dy \nonumber   \\ 
		&  +  2	\int_{a(\Omega)}  \big[ 	   A (y, x) G[f\circ a^{-1}](y)    -     G_r[g\circ a^{-1}](y) A (x,y)  \big]
		dy \nonumber \\
		= &       \left\{ \begin{array}{ll} f\circ a^{-1}(x) + g\circ a^{-1}(x), &  x\in a(\Omega), \\ 0 ,&  x\in \mathbb H\setminus\overline{a(\Omega)} .  \end{array}       \right.    
	\end{align*}
	Then change of  variable $y=a(q)$ gives us that 
	\begin{align*}   & \displaystyle \int_{\partial  \Omega }  \|\underline a(\underline \tau)  \|^2  \big[  {g (\tau)} J_{a}\nu_{a(\tau)}   A ( x, a(\tau))  - A (a(\tau), x) J_a \nu_{a(\tau)}  f (\tau) \big]  \nonumber \\
		& +  		\int_{ \Omega } \big[ B (a(y), x)f (y)  -     
		g (y) C (x,a(y))	\big]
		M_a	dy  \nonumber   \\ 
		&  +  2	\int_{ \Omega }  \big[ 	   A (a(y), x)  G[f\circ a^{-1}] \circ a (y)    -   G_r[g\circ a^{-1}]\circ a (y) A (x,a(y))  \big]
		M_a dy \nonumber \\
		= &       \left\{ \begin{array}{ll} f\circ a^{-1}(x) + g\circ a^{-1}(x), &  x\in  a(\Omega ), \\ 0 ,&  x\in \mathbb H\setminus\overline{a( \Omega )}.  \end{array}       \right.    
	\end{align*} 
	Finally, write   $a(x)$ for   $x\in  \Omega$, then    
	\begin{align*}   & \displaystyle \int_{\partial  \Omega }  \|\underline{a(\underline\tau)} \|^2  \big[  {g (\tau)} J_{a}\nu_{a(\tau)}   A ( a(x), a(\tau))  - A (a(\tau), a(x)) J_a \nu_{a(\tau)}   f (\tau) \big]  \nonumber \\
		& +  		\int_{ \Omega } \big[ B (a(y), a(x))f (y)  -     
		g (y) C (a(x),a(y))	\big]
		M_a	dy  \nonumber   \\ 
		&  +  2	\int_{ \Omega }  \bigg[ 	   A (a(y), a(x))  G[f\circ a^{-1}] \circ a (y)  \\
		& \qquad   -   G_r[g\circ a^{-1}]\circ a (y) A (a(x),a(y))  \bigg]
		M_a dy \nonumber \\
		= &       \left\{ \begin{array}{ll} f (x) + g (x), &  x\in    \Omega  , \\ 0 ,&  x\in \mathbb H\setminus\overline{ ( \Omega) },  \end{array}       \right.    
	\end{align*}
	Finally, use 
	$$   {}^{-\underline{a}}M\circ \mathcal H_a (f) =   W_{a } \circ  G   \circ W_{a^{-1}}  [f]  , 
	\quad 	{}^{-\underline{a}}M\circ \mathcal H_{a,r} (g) =    W_{a } \circ  G_r   \circ W_{a^{-1}}  [g]  $$
of 	Proposition \ref{HGQuaternions}.
\end{proof}
\begin{cor} \label{cor_quaternionic_Borel-Pompeiu_form_L_a} (Cauchy-type formula  in $\textrm{Ker}(\mathcal H_a)$ and $\textrm{Ker}(\mathcal H_{a,r})$).
According to the hypothesis of Theorem \ref{quaternionic_Borel-Pompeiu_form_La} and if $f\in \textrm{Ker}(\mathcal H_a)\cap  C^{1}(\overline{ \Omega },\mathbb H)$ and	$g\in \textrm{Ker}(\mathcal H_{a,r})\cap  C^{1}(\overline{ \Omega },\mathbb H)$, then 
	\begin{align*}   & \displaystyle \int_{\partial  \Omega }  \|\underline{a}(\underline\tau)  \|^2  \big[  {g (\tau)} J_{a}\nu_{a(\tau)}   A ( a(x), a(\tau))  - A (a(\tau), a(x)) J_a \nu_{a(\tau)}   f (\tau) \big]  \nonumber \\
		& +  		\int_{ \Omega } \big[ B (a(y), a(x))f (y)  -     
		g (y) C (a(x),a(y))	\big]
		M_a	dy  \nonumber   \\ 
		= &       \left\{ \begin{array}{ll} f (x) + g (x), &  x\in  ( \Omega) , \\ 0 ,&  x\in \mathbb H\setminus\overline{  ( \Omega) }.  \end{array}       \right.    
	\end{align*}
\end{cor}

\begin{thm} \label{quaternionic_Borel-Pompeiu_form_D}(Borel-Pompeiu-type formula associated to $  {D}_u$).	
Suppose   the hypothesis and notation of Theorem \ref{quaternionic_Borel-Pompeiu_form_La}. If
	$u_i =     \sum_{j=1}^3  a _j   \frac{\partial (a^{-1})_i  }{\partial  y_j}\circ    a    $ for $i=1,2, 3$, then
	\begin{align*}   & \displaystyle \int_{\partial  \Omega }  \|\underline a(\underline \tau)  \|^2  \big[  {g (\tau)} J_{a}\nu_{a(\tau)}  A ( a(x), a(\tau))  - A (a(\tau), a(x)) J_a \nu_{a(\tau)}  f (\tau) \big]  \nonumber \\
		& +  		\int_{ \Omega } \big[ B (a(y), a(x))f (y)  -     
		g (y) C (a(x),a(y))	\big]
		M_a	dy  \nonumber   \\ 
		&  +  2	\int_{ \Omega }  \big[ 	   A (a(y), a(x)) 
		\underline{a}(\underline y) (1  +{\underline a} (\underline y) )
		{\frac {\partial f}{\partial y_0}}(y)
		\\
		& \qquad + 
		{\frac {\partial f}{\partial y_0}}(y) (1  +{\underline a}(\underline y) )  \underline{a}(\underline y)   A (a(x),a(y))  \big]
		M_a dy \nonumber \\
		&  +  2	\int_{ \Omega }  \big[ 	   A (a(y), a(x)) 
		\underline{a}(\underline y) {  {D}_u f}(y)    + 
		{ {D}_u g}(y) \underline{a}(\underline y)   A (a(x),a(y))  \big]
		M_a dy \nonumber \\
		= &       \left\{ \begin{array}{ll} f (x) + g (x), &  x\in  ( \Omega) , \\ 0 ,&  x\in \mathbb H\setminus\overline{  ( \Omega) },  \end{array}       \right.    
	\end{align*}
	for all  $f ,g  \in C^{1}(\overline{ \Omega },\mathbb H)$.	
\end{thm}
\begin{proof}
	Use the identities  
	\begin{align}\label{HDquaternionic}
		\mathcal H_a(f)  =  & (1  +{\underline a} ) 
		{\frac {\partial f}{\partial y_0}}
		-  {  {D}_u f}  , \nonumber \\ 
	\mathcal 	H_{a,r}(g)  =  &  
		{\frac {\partial g}{\partial y_0}} (1  +{\underline a} )
		- {  {D}_u g},   
	\end{align} 
	deduced from Remark \ref{remDHQuater} and also use Theorem \ref{quaternionic_Borel-Pompeiu_form_La}. 
\end{proof}

\begin{cor}\label{quaternionic_Cauchy_form_D} (Cauchy-type formula associated to $  {D}_u$).  According to  hypothesis and notation of Theorem \ref{quaternionic_Borel-Pompeiu_form_La}. Also suppose that 
	$u_i =     \sum_{j=1}^3  a _j   \frac{\partial (a^{-1})_i  }{\partial  y_j}\circ    a    $ for $i=1,2, 3$. 
  If $f ,g  \in \textrm{Ker}( {D}_u) \cap C^{1}(\overline{ \Omega },\mathbb H)$, then 
	\begin{align*}   
		& \displaystyle \int_{\partial  \Omega }  \|\underline a(\underline\tau) \|^2  \big[  {g (\tau)} J_{a}\nu_{a(\tau)}   A ( a(x), a(\tau))  - A (a(\tau), a(x)) J_a \nu_{a(\tau)}   f (\tau) \big]  \nonumber \\
		& +  		\int_{ \Omega } \big[ B (a(y), a(x))f (y)  -     
		g (y) C (a(x),a(y))	\big]
		M_a	dy  \nonumber   \\ 
		&  +  2	\int_{ \Omega }  \big[ 	   A (a(y), a(x)) 
		\underline{a}(\underline y) (1  +{\underline a} (\underline y)) 
		{\frac {\partial f}{\partial y_0}}(y)
		\\
		& \qquad  + 
		{\frac {\partial g}{\partial y_0}}(y) (1  +{\underline a} (\underline y))  \underline{a}(\underline y)   A (a(x),a(y))  \big]
		M_a dy \nonumber \\
		= & \left\{ \begin{array}{ll} f (x) + g (x), &  x\in  ( \Omega) , \\ 0 ,&  x\in \mathbb H\setminus\overline{(\Omega)}. 
		\end{array} \right.    
	\end{align*}
\end{cor} 
\begin{prop}\label{prop3.8D} (Stokes' theorem associated to $ \mathcal H_a$  and $  \mathcal  H_{a,r}$).
	Let $\Omega\subset (s,t) \times U$ be a domain such that $\overline{a(\Omega)}\subset (s,t)\times V  $ and  $\partial ( \Omega )$  
	and $\partial ( a(\Omega))$ are $3-$dimensional compact smooth surfaces  and $\overline{a(\Omega)}\subset \mathbb H\setminus \mathbb R$. Then
	\begin{align*} & \displaystyle  \int_{\partial \Omega } g J_a \nu_{a(x)} f  = 4 \int_{ \Omega } g  \frac{\underline a(\underline x )}{\|\underline a(\underline x) \|^2} f M_a dx \\ 
		&- 2 \int_{ \Omega }  \frac{1}{\|\underline a(\underline x) \|^2}   \left[ 
		\mathcal H_{a,r}(g) (x)  \underline a (\underline x) f(x)   +    g(x)  \underline a(\underline x)   \mathcal H_a(f)(x)  \right] M_a dx  
	\end{align*}
	for all  $f ,g  \in C^{1}(\overline{ \Omega },\mathbb H)$.	
\end{prop}
\begin{proof}Apply Proposition \ref{StokesG}  	to funtions   $f\circ a^{-1},g\circ a^{-1} \in C^{1}(\overline{a(\Omega)},\mathbb H)$  
	and use \eqref{HDquaternionic}  to obtain  that 
	\begin{align*} & \displaystyle  \int_{\partial a(\Omega)} g\circ a^{-1} \nu_x  f\circ a^{-1} = 4 \int_{a(\Omega)} g\circ a^{-1}  \frac{\vec x }{\|\vec x \|^2} f\circ a^{-1}(x) dx \\ 
		&+ 2 \int_{a(\Omega)}  \frac{1}{\|\underline x \|^2}   \left(    G_r[g\circ a^{-1}] (x)f\circ a^{-1}(x)  +    g\circ a^{-1} (x) G[f\circ a^{-1}](x)  \right)dx  
	\end{align*}
	and change of variables    $ a(x)$ instead of $x$  to  obtain 
	\begin{align*} & \displaystyle  \int_{\partial  \Omega } g J_a \nu_{a(x)}  f 
		= 4 \int_{ \Omega } g  \frac{\underline a(\underline x )}{\|\underline  a(\underline x) \|^2} f(x) M_a dx \\ 
		&- 2 \int_{ \Omega }  \frac{1}{\|\underline  a(\underline x) \|^2}   \left[ 
		H_{a,r}(g)  (x) \underline a(\underline x)  f (x)  +    g(x)  \underline a(\underline x)   H_a(f)(x)  \right] M_a dx  .
	\end{align*}
\end{proof}

\begin{cor}\label{corstokesH} (Cauchy Theorem) According to the hypothesis of the previous proposition. If 
	$f \in \textrm{Ker}(\mathcal H_a)\cap  C^{1}(\overline{ \Omega },\mathbb H)$
	and $ g\in   \textrm{Ker}(\mathcal H_{a,r}) \cap   \in C^{1}(\overline{ \Omega },\mathbb H)$,
	then 
	\begin{align*} & \displaystyle  \int_{\partial \Omega } g J_a \nu_{a(x)} f  = 4 \int_{ \Omega } g  \frac{\underline a(\underline x )}{\|\underline a(\underline x) \|^2} f M_a dx . 
	\end{align*}
\end{cor}

\begin{cor} (Stokes' theorem associated to $ {D}_u$). Suppose the hypothesis and notation  of Proposition \ref{prop3.8D}.   If
	 $u_i =    \sum_{j=1}^3 a_j    \frac{\partial (a^{-1})_i}{\partial  y_j}\circ    a   $ 
for $i=1,\dots, 3$. Then
	\begin{align*} & \displaystyle  \int_{\partial \Omega } g J_a \nu_{a(x)}  f  = 4 \int_{ \Omega } g  \frac{\underline a(\underline x )}{\|\underline a(\underline x) \|^2} f(x) M_a dx \\ 
		&- 2 \int_{ \Omega }  \frac{1}{\|\underline a(\underline x) \|^2}   \bigg[ 
		{\frac {\partial g}{\partial y_0}}(x) (1  +{\underline a} (\underline  x))
		\underline a(\underline  x)  f (x)  \\
		& \qquad  +    g(x)  \underline a (\underline x)  (1  +{\underline a}(\underline  x) ) 
		{\frac {\partial f}{\partial y_0}}(x)
		\bigg] M_a dx  \\
		&+ 2 \int_{ \Omega }  \frac{1}{\|\underline a(\underline x) \|^2}   \bigg[ 
		{ {D}_u g}(x)  \underline a(\underline x)  f (x)  +    g (x) \underline a (\underline x )  {  {D}_u f (x)}  \bigg] M_a dx  
	\end{align*}
	for all  $f ,g  \in C^{1}(\overline{ \Omega },\mathbb H)$.	
\end{cor}
\begin{proof}
Use  \eqref{HDquaternionic} in Proposition \ref{prop3.8D}.
\end{proof}

\begin{cor}   
	Suppose the hypothesis   of Proposition \ref{prop3.8D}. If 
	 $u_i =   \sum_{j=1}^3 a_j (\underline{x}) \frac{\partial (a^{-1})_i  }{\partial  y_j}\circ a  $,  
for $i=1,2, 3$, and  $  {D}_u g =  {D}_u f =0$ on $\Omega$, then
	\begin{align*} & \displaystyle  \int_{\partial \Omega } g J_a \nu_{a(x)} f = 4 \int_{\Omega} g \frac{\underline a(\underline x )}{\|\underline a(\underline x) \|^2} f (x)M_a dx \\ 
		&- 2 \int_{ \Omega }  \frac{1}{\|\underline a(\underline x) \|^2}   \bigg[ 
		{\frac {\partial g}{\partial y_0}}(x) (1  +{\underline a} (\underline x))
		\underline a(\underline x)  f(x)   +    g (x) \underline a (\underline x)  (1  +{\underline a}(\underline  x) ) 
		{\frac {\partial f}{\partial y_0}}(x)
		\bigg] M_a dx , 
	\end{align*}
	for all  $f ,g  \in C^{1}(\overline{ \Omega },\mathbb H)$.	
\end{cor}

\begin{rem}
Borel-Pompieu, Stokes and Cauchy type formulas  deeply related with  the solution of some Dirichlet type boundary-value problems.  Therefore, although it is not our objective in this work, the solution to some Dirichlet type boundary-value problems   induced by $\mathcal H_a$ and $D_u$ can be proposed.
\end{rem}

\begin{prop}\label{ConfCovPropL}(Conformal covariant property of $\mathcal H_a$). 
	Let $\Omega\subset \mathbb H$ be a domain and  
	let $T$ be a quternionic M\"obius transformation given by 
	Proposition \ref{ConfCovPropG}  such that   $\Omega, T (\Omega) \subset (s,t)\times V $. 
	Then  
	\begin{align*}
		{}^{\underline{I} }M \circ W_{a^{-1}} \circ  \mathcal  H_a \circ W_a \circ {}^{\mathcal  A_{\mathcal T} } M \circ W_{ \mathcal T }   =   {}^{\mathcal  B_{\mathcal T}     \underline{\mathcal T} }M \circ W_{\mathcal T} \circ  \mathcal H_a\circ W_a,    
	\end{align*}
	on $C^1({\mathcal T} (\Omega),\mathbb H)$, where $\underline I(x)= \underline x$ for all $x$
 	 and  $\mathcal A_{\mathcal T}$ and $\mathcal B_{\mathcal T}$ are  given by \eqref{functions}. 
\end{prop}
\begin{proof}
	From Propositions \ref{HGQuaternions} and  \ref{ConfCovPropG} we obtain that  
	$
		  G=    W_{a^{-1}}\circ {}^{-\underline{a}}M \circ\mathcal H_a \circ W_a$   and 
	\begin{align*}
		  W_{a^{-1}}\circ {}^{-\underline{a}}M \circ \mathcal H_a \circ W_a \circ ({}^{\mathcal A_{\mathcal T}}M\circ W_{\mathcal T})  
		& =   ({}^{\mathcal B_{\mathcal T}}M\circ W_{\mathcal T})\circ W_{a^{-1}}\circ {}^{-\underline{a}}M \circ \mathcal H_a \circ W_a, 
	\end{align*}
	on 	    $ C^1( (s,t)\times V ,\mathbb H) $. 
	The second identity is equivalent to
	\begin{align*}
		   - \underline{q} \mathcal H_a [ (\mathcal A\circ a ) (g\circ\mathcal T \circ a) ]\circ a^{-1} (q)   
		= & \mathcal {B_{\mathcal T}} (q) \ \underline{\mathcal T}(q)\ 
		( \mathcal H_a [ g\circ a]\circ a^{-1})\circ \mathcal T(q), \quad \forall q\in \Omega, 
	\end{align*}
	for all   $g \in C^1({\mathcal T}(\Omega),\mathbb H)$.
\end{proof}

\begin{cor}\label{conformkerH}According to the previous proposition we see   that  

$f\in \textrm{Ker}(\mathcal H_a ) \cap C^1(a^{-1}( \mathcal T(\Omega)), \mathbb H)$ if and only if $(\mathcal A\circ a ) (f\circ a^{-1}\circ \mathcal T \circ a) \in \textrm{Ker} (\mathcal H_a)$.
\end{cor}

\begin{cor}\label{ConfCovPropDDD} (Conformal properties of $  {D}_u$).  
	Consider all hypothesis of Proposition \ref{ConfCovPropL}. If $u_i = \sum_{j=1}^3 a_j  \frac{\partial (a^{-1})_i}{\partial y_j}\circ a $ for $i=1,2, 3$ and $f \in C^1(a^{-1}({\mathcal T}(\Omega)),\mathbb H)$. Then  
	\begin{align*}
		& - \underline{q} (1  +{\underline q} ) 
		{\frac {\partial  (\mathcal A\circ a ) (f\circ a^{-1}\circ\mathcal T \circ a)}{\partial y_0}}\circ a^{-1} (q)  \\
		& \ \   + \underline{q}{ {D}_u [ (\mathcal A\circ a ) (f\circ a^{-1}\circ\mathcal T \circ a)]  \circ a^{-1} (q)  } \\
		= & \mathcal {B_{\mathcal T}} (q) \ \underline{\mathcal T}(q)\ 
		(1  +{\underline  T}(q) ) 
		{\frac {\partial f}{\partial y_0}}\circ a^{-1}   \circ \mathcal T(q) 
		-  \mathcal {B_{\mathcal T}} (q) \ \underline{\mathcal T}(q) {  {D}_u [f]}\circ a^{-1}   \circ \mathcal T(q) 
	\end{align*}
	for all $q\in \Omega$.
\end{cor}
\begin{proof}
	Use Proposition \ref{ConfCovPropL} and  \eqref{HDquaternionic} 
	to obtain 
	\begin{align*}
		& - \underline{q} (1  +{\underline q} ) 
		{\frac {\partial  (\mathcal A\circ a ) (g\circ\mathcal T \circ a)}{\partial y_0}}\circ a^{-1} (q) 
		+ \underline{q}{ {D}_u [ (\mathcal A\circ a ) (g\circ\mathcal T \circ a)]  \circ a^{-1} (q)  } \\
		= & \mathcal {B_{\mathcal T}} (q) \ \underline{\mathcal T}(q)\ 
		(1  +{\underline  T}(q) ) 
		{\frac {\partial g\circ a}{\partial y_0}}\circ a^{-1}   \circ \mathcal T(q) 
		\\
		&  \ \  -  \mathcal {B_{\mathcal T}} (q) \ \underline{\mathcal T}(q) {  {D}_u [g\circ a]}\circ a^{-1}   \circ \mathcal T(q) 
	\end{align*}
	for all $q\in \Omega$ for all   $g \in C^1({\mathcal T}(\Omega),\mathbb H)$. 
\end{proof}

\begin{cor} (Invariant conformal property of  Ker$(  {D}_u )$).  
	Consider all hypothesis of Proposition \ref{ConfCovPropL}. Denote  $u_i =   \sum_{j=1}^3 a_j   \frac{\partial (a^{-1})_i}{\partial y_j}\circ a $ for $i=1,2, 3$. If $f \in \textrm{Ker}(  {D}_u) \cap C^1(a^{-1}({\mathcal T}(\Omega)),\mathbb H)$  then  
	\begin{align*}
		&  \underline{q}{ {D}_u [ (\mathcal A\circ a ) (f\circ a^{-1}\circ\mathcal T \circ a)]  \circ a^{-1} (q)  } \\
		= & \mathcal {B_{\mathcal T}} (q) \ \underline{\mathcal T}(q)\ 
		(1  +{\underline  T}(q) ) 
		{\frac {\partial f}{\partial y_0}}\circ a^{-1}   \circ \mathcal T(q) \\
		& +   \underline{q} (1  +{\underline q} ) 
		{\frac {\partial  (\mathcal A\circ a ) (f\circ a^{-1}\circ\mathcal T \circ a)}{\partial y_0}}\circ a^{-1} (q) , 
	\end{align*}
	for all $q\in \Omega$.
\end{cor}

\begin{rem}
The conformal properties of $G$, $\mathcal H_a$ and $D_u$ are   certain generalizations of the well-known chain rule  of complex  analysis. Furthermore, these properties often induce some isomorphism between function modules induced by these operators and associated with conformally equivalent domains.
\end{rem}

\begin{rem}
	The  operator $\mathcal H_{a,r}$  has similar properties to $\mathcal  H_a$ and the properties it helps us discover about the $  {D}_u$ are deeply similar, so the right case is omitted.
	
	On the other hand, if $u_i =  \sum_{j=1}^3 a _j   \frac{\partial (a^{-1})_i}{\partial y_j}\circ a  $ for $i=1,2, 3$. Then 
	$\mathcal H_a(f) = 0$ if and only if ${  {D}_u f} = (1  +{\vec a}) {\frac {\partial f}{\partial y_0}}$, on $(s,t)\times U$. Therefore all properties given above about  $\textrm{Ker}(\mathcal H_a)$ under hypothesis  
	$$u_i =   \sum_{j=1}^3 a_j   \frac{\partial (a^{-1})_i}{\partial y_j}\circ a , \quad \textrm{for} \quad i=1,2, 3$$ are the same properties of the quaternionic right module  $ \mathcal  D$   formed by  $f\in C^{1}((s,t)\times U, \mathbb H)$
	such that $ { {D}_u f} = (1 +{\vec a}) {\frac {\partial f}{\partial y_0}}$ on $(s,t)\times U$, i.e., the  quaternionic right module $\mathcal D$ has deeply similar properties to commented in Propositions \ref{HGQuaternions}, \ref{SerKerHa}, Theorem \ref{thmKerHASSD}, and Corollaries \ref{cor_quaternionic_Borel-Pompeiu_form_L_a}, \ref{corstokesH}, \ref{conformkerH}. 
\end{rem}

\section{Conclusions and future works}
The study of the $G$ operator has been generalized via the theory found for $\mathcal H_a$  and that in turn under certain relations we can obtain information about the material derivative ${ D}_u$ either in the quaternionic case or in Clifford analysis.

It is important to note that the results of the quaternionic analysis on the operator $G$ can be discovered in Clifford analysis in a very similar way and thus find properties of the n-dimensional material  derivative that will be similar to the quaternionic material derivative.

We shall focus on describing important properties of $\mathcal H_a$ and ${  D}_u$ in order to develop future applications of their properties in some differential equations with many applications to physics, among other phenomena.  

Another important future work is the   development of the theory of functions induced by the operators of the form  $H_{a,b}$ and with it a generalization of the theory of slice monogeinc functions among other.


\begin{thebibliography}{99}
\bibitem{A1} D. Alpay, K. Diki, I. Sabadini. \textit{On the global operator and Fueter mapping theorem for slice polyanalytic functions}. Anal. Appl. (Singap.) 19, no. 6, 941--964,  (2021) 

\bibitem{A2} D. Alpay, K. Diki, M. Vajiac. \textit{New Fueter-type variables associated to the global operator in the quaternionic case}. Publ. Res. Inst. Math. Sci. 60, no. 4, 859--891, (2024) 

\bibitem{A3} D. Alpay, I. Cho, M. Vajiac. \textit{Scaled global operators and Fueter variables on non-zero scaled hypercomplex numbers}. Adv. Appl. Clifford Algebr. 34, no. 5, Paper No. 53, 61 pp.  (2024)

\bibitem{B} A.V. Baev. \textit{Solving the Navier-Stokes equation for a viscous incompressible fluid in an $n$-dimensional bounded region and in the entire space $\mathbb R^n$}. Comput. Math. Model. 33, no. 3, 255--272,  (2022)

\bibitem{Bat} G. K. Batchelor. \textit{An Introduction to Fluid Dynamics}. Cambridge University Press, pp. 72-73, (1967)

\bibitem{cgs} F. Colombo, G. Gentili,  I. Sabadini, D. Struppa. \textit{Extension results for slice regular functions of a quaternionic variable}. Adv. Math. 222, no. 5, 1793--1808,  (2009)

\bibitem{CSS} F. Colombo, I. Sabadini, D.C. Struppa, \textit{Noncommutative functional calculus. Theory and applications of slice hyperholomorphic functions}. Progress in Mathematics, 289. Birkhäuser/Springer Basel AG, Basel, vi+221 pp. (2011)

\bibitem{CZ} C. Ding, Z. Xu. \textit{Invariance of iterated global differential operator for slice monogenic functions}. Comput. Methods Funct. Theory 25, no. 3, 735--752, (2025)

\bibitem{GS} G. Gentili, D.C. Struppa, \textit {A new theory of regular functions of a quaternionic variable}. Adv. Math. 216, no. 1, 279--301,  (2007)

\bibitem{CS} F. Colombo, F. Sommen. \textit {Distributions and the global operator of slice monogenic functions}. Complex Anal. Oper. Theory 8, no. 6, 1257--1268,  (2014) 

\bibitem{GlobalOp} F. Colombo, J. O. Gonz\'alez-Cervantes, I. Sabadini. \textit{A nonconstant coefficients differential operator associated to slice monogenic functions}. Trans. Amer. Math. Soc. 365, no. 1, 303--318,  (2013) 

\bibitem{CSS9} F. Colombo, I. Sabadini, D.C. Struppa. \textit{Slice monogenic functions}. Israel J. Math. 171, 385--403,  (2009) 

\bibitem{CC} C. Sydney, T. G., Cowling. \textit{The mathematical theory of nonuniform gases. An account of the kinetic theory of viscosity, thermal conduction and diffusion in gases}. In co-operation with D. Burnett. Revised reprint of the third edition. With a foreword by Carlo Cercignani. Cambridge Mathematical Library. Cambridge University Press, Cambridge, (1990) 

\bibitem{F} H. G. Fortak. \textit{Material derivatives of higher dimension in geophysical fluid dynamics}. Meteorologische Zeitschrift Vol. 13 No. 6 , p. 499-510, (2004)

\bibitem{G} J. O. Gonz\'alez-Cervantes. \textit{On Cauchy integral theorem for quaternionic slice regular functions}. Complex Anal. Oper. Theory 13, no. 6, 2527--2539,  (2019) 

\bibitem{GG1} J. O. Gonz\'alez-Cervantes, D. Gonz\'alez-Campos. \textit{The global Borel-Pompieu-type formula for quaternionic slice regular functions}. Complex Var. Elliptic Equ. 66, no. 5, 721--730,  (2021) 

\bibitem{GG2} J. O. González-Cervantes, D. González-Campos. \textit{On the conformal mappings and the global operator $G$}. Adv. Appl. Clifford Algebr. 31, no. 1, Paper No. 6, 13 pp.  (2021) 

\bibitem{GPG} J. O. González-Cervantes, J. E. Paz-Cordero, D. González-Campos. \textit{On some quaternionic series}. Adv. Appl. Clifford Algebr. 33, no. 4, Paper No. 45, 17 pp.  (2023)

\bibitem{gp} R. Ghiloni, A. Perotti, \textit{Volume Cauchy formulas for slice functions on real associative *-algebras}. Complex Var. Elliptic Equ. 58, no. 12, 1701--1714,  (2013)

\bibitem{gpr} R. Ghiloni, A. Perotti, V. Recupero, \textit{Noncommutative Cauchy integral formula}. Complex Anal. Oper. Theory 11, no. 2, 289--306,  (2017)  

\bibitem{GR}  R. A. Granger, R.A. \textit{Fluid Mechanics}. Courier Dover Publications. (1995) 

\bibitem{GS1}  K. G\"urlebeck, W. Spr\"ossig, \textit{Quaternionic analysis and elliptic boundary value problems.} Birkha\"user Verlag, (1990)

\bibitem{GS2}  K. G\"urlebeck, W. Spr\"ossig, \textit{Quaternionic and Clifford calculus for physicists and engineers}. John Wiley and Sons. (1997)

\bibitem{gssbook} G. Gentili, C Stoppato, D.C. Struppa. \textit {Regular functions of a quaternionic variable}, Springer Monographs in Mathematics, Springer, Berlin-Heidelberg, (2013)

\bibitem{GP} R. Ghiloni, A. Perotti. \textit {Global differential equations for slice regular functions}. Math. Nachr. 287, no. 5-6, 561--573,  (2014)

\bibitem{OHJ} H. Ockendon, J.R. Ockendon. \textit {Waves and compressible flow}. Texts in Applied Mathematics, 47. Springer-Verlag, New York, (2004)	

\bibitem{PK}  W.J. Prosnak, Z.J. Kosma. \textit{On a new method for numerical solution of the Navier-Stokes equations}. Acta Mech. 89, no. 1-4, 45--63,  (1991)

\bibitem{Z} L. Zhang. \textit{Properties of solutions of $n$-dimensional incompressible Navier-Stokes equations}. Ann. Appl. Math. 35, no. 4, 392--448,  (2019)

\end{thebibliography}
\end{document}